\documentclass[11pt,reqno]{amsart}
    \usepackage{amssymb}
    \usepackage{amsmath}
    \usepackage{mathrsfs}
    \usepackage{amssymb, amsmath, amsfonts}
    \usepackage[title]{appendix}
    \usepackage{mathabx}
   \usepackage{color}

    \usepackage[hidelinks]{hyperref}
    \allowdisplaybreaks
    \usepackage[numbers,sort&compress]{natbib}

    \makeatletter
    \def\tank#1{\mathbb Protected@xdef\@thanks{\@thanks
     \mathbb Protect\footnotetext[0]{#1}}}
    \def\bigfoot{

     \@footnotetext}
    \makeatother

    \newcommand{\ea}{\end{array}}

    \allowdisplaybreaks
    \numberwithin{equation}{section}

    \allowdisplaybreaks

    \newtheorem{theorem}{Theorem}[section]
    \newtheorem{lemma}{Lemma}[section]
    \newtheorem{proposition}[theorem]{Proposition}

    \newtheorem{condition}[theorem]{Condition}
    \newtheorem{corollary}[theorem]{Corollary}
    \newtheorem{definition}[theorem]{Definition}
    
    \newtheorem{remark}{Remark}[section]

    \def\beq{\begin{equation}}
    \def\nneq{\end{equation}}

    \def\bthm{\begin{theorem}}
    \def\nthm{\end{theorem}}

    \def\blem{\begin{lemma}}
    \def\nlem{\end{lemma}}
    \def\bprf{\begin{proof}}
    \def\nprf{\end{proof}}
    \def\bprop{\begin{prop}}
    \def\nprop{\end{prop}}
    \def\brmk{\begin{rem}}
    \def\nrmk{\end{rem}}

    \def\bexa{\begin{exa}}
    \def\nexa{\end{exa}}
    \def\bcor{\begin{cor}}
    \def\ncor{\end{cor}}

        \newcommand{\E}{\mathbb E}
 
  \newcommand{\eps}{\varepsilon}
      \def\R{\mathbb{R}}
    \def\RR{\mathbb{R}}
    
    \def\EE{\mathbb{E}}

    \def\cF{\mathcal{F}}
    
    \def\cD{\mathcal{D}}

    \def\cN{\mathcal{N}}

    \def\ee{{\mathbb E}}
    \def\e{{\varepsilon}}

    \newcommand{\1}{{\bf 1}}

    \newcommand{\lc}{\left(}
    \newcommand{\rc}{\right)}
    \newcommand{\lk}{\left[}
    \newcommand{\rk}{\right]}
    \newcommand{\lt}{\left }
    \newcommand{\rt}{\right}

    \title[Temporal regularity for the nonlinear SHE]{Temporal regularity  for the nonlinear stochastic  heat equation  with spatially rough noise}

         \date{}
    
\begin{document}

       \author[B. Qian]{Bin Qian}
    \address[]{Bin Qian, Department of Mathematics and Statistics, Suzhou university of Technology, Changshu, Jiangsu, 215500,  China.}
    \email{binqiancn@126.com}

    \author[M. Wang]{Min Wang}
    \address[]{Min Wang, School of Mathematics and Statistics,  Wuhan University of Technology,  Wuhan, 430063, 
    China.}
    \email{minwangmath@whut.edu.cn}

    \author[R. Wang]{Ran Wang}
    \address[]{Ran Wang, School of Mathematics and Statistics,  Wuhan University,  Wuhan, 430072,
    China.}
    \email{rwang@whu.edu.cn}

    \author[Y. Xiao]{Yimin Xiao}

    \address[]{Yimin Xiao, Department of Statistics and Probability, Michigan State University, East Lansing,
    MI 48824, USA.}
    \email{xiaoy@msu.edu}

    \maketitle
     \noindent {\bf Abstract:}
   Consider  the nonlinear stochastic  heat equation 
$$
  \frac{\partial u (t,x)}{\partial t}=\frac{\partial ^2 u (t,x)}{\partial x^2}+ \sigma(u (t,x))\dot{W}(t,x),\quad   t> 0,\,
   x\in \mathbb{R},
$$
where $\dot W$ is a Gaussian noise which is white in time and fractional  in space with Hurst parameter $H\in(\frac 14,\frac 12)$.  
The existence and uniqueness of the solution of this equation was proved by Balan et al. \cite{BJQ2015} when 
$\sigma(u) = au + b$ is an affine function, and by Hu et al.  \cite{HHLNT2017} when $\sigma$ is differentiable with Lipschitz derivative 
and $\sigma(0)=0$.  In these cases, the H\"older continuity of the solution  
has also been established by Balan  et al. \cite{BJQ2016} and Hu et al.  \cite{HHLNT2017}, respectively.

In this paper,  we study the asymptotic properties of the temporal gradient $u(t+\varepsilon, x)-u(t, x)$ at any fixed $t \ge  0$ and 
$x\in \mathbb R$,  as $\varepsilon\downarrow 0$, under the framework of \cite{HHLNT2017}. As applications, we deduce  
Khinchin's law of the iterated logarithm,  Chung's 
law of the iterated logarithm, and  the  quadratic variation  of the temporal process $\{u(t, x)\}_{t \ge 0}$, where $x\in \mathbb R$ is fixed.

 \vskip0.3cm

    \section{Introduction  }

    In this paper, we are interested in the temporal regularity of the solution to the nonlinear stochastic  heat equation  (SHE, for short):  
    \begin{equation}\label{SHE rough}
      \frac{\partial  u(t,x)}{\partial t}=\frac{\partial ^2  u(t,x)}{\partial x^2}+ \sigma(u(t,x))\dot{W}(t,x),\quad   t>0,\,
       x\in\RR,
    \end{equation}
where $W(t,x)$ is a centered Gaussian field with the covariance given by
    \begin{equation}\label{CovW}
      \ee\left[W(t,x)W(s,y)\right]=\frac 12 \left(s\wedge t\right)\left( |x|^{2H}+|y|^{2H}-|x-y|^{2H} \right),
    \end{equation}
with $\frac 14<H<\frac 12$. That is, $W$ is a standard Brownian motion in time and a fractional Brownian motion (fBm, for short) 
with Hurst index $H$ in space and $\dot W(t,x)=\frac{\partial ^2}{\partial t\partial x}W(t,x)$. Formally,  the covariance of the noise 
$\dot{W}$ is given by
    $$\ee\left[\dot{W}(t,x)\dot{W}(s,y)\right]=\delta_0 ( t-s)\Lambda\left(x-y\right),$$
    where the spatial covariance $\Lambda$ is a distribution, whose Fourier transform is the measure    
    \begin{align}\label{e.mu}
    \mu(d\xi)=c_{1,1}|\xi|^{1-2H}d\xi,
    \end{align}
     with
     \begin{align}
       c_{1,1}=&\,\frac{1}{2\pi}\Gamma\left(2H+1\right)\sin\left(\pi H\right).   \label{e.c1}
       \end{align}
   The spatial covariance $\Lambda(x-y)$ can be formally written as $\Lambda(x-y)=H(2H-1) |x-y|^{2H-2}$. However, the corresponding 
   covariance $\Lambda$ is not locally integrable and does not become nonnegative when $H\in(\frac 14,\frac 12)$.  It does not satisfy 
   the classical Dalang's condition in \cite{Dalang1999}, where $\Lambda$ is given by a non-negative locally integrable function.  
   Consequently, the standard approaches used in references \cite{Dalang1999, DQ, DaPrato2014} do not apply to such rough covariance 
   structures.

         Recently, many authors have studied the existence and uniqueness of solutions of stochastic partial differential equations (SPDEs, for short) driven by Gaussian noise with the covariance of a fractional Brownian motion with Hurst parameter $H\in (\frac1 4, \frac12)$ in the space variable. See, e.g., \cite{HHLNT2017}, \cite{HHLNT2018}, \cite{HW2022}, \cite{LHW2022}, and \cite{SongSX2020}. For surveys on the subject, we refer to \cite{hu19} and \cite{Song2018}. When the diffusion coefficient is affine, i.e., $\sigma(x)=ax+b$, Balan et al.   \cite{BJQ2015} proved the existence and uniqueness of the mild solution to  SHE \eqref{SHE rough} using the Fourier analytic techniques.  They also established the H\"older continuity of the solution in \cite{BJQ2016}.  In the case of a non-linear coefficient $\sigma(u)$, Hu et al. \cite{HHLNT2017}  proved the well-posedness of SHE \eqref{SHE rough} under the
assumption that $\sigma(u)$ is Lipschitz continuous,  differentiable with a Lipschitz derivative and that $\sigma(0)=0$.  Under similar conditions,  Liu and Mao \cite{LM2022}
studied the well-posedness and intermittency of the SHE with non-local fractional differential operators.  Hu and Wang \cite{HW2022} removed the condition $\sigma(0)=0$ by introducing a decay weight.

 When $W$ is space-time white noise, i.e. $H=\frac{1}{2}$, Khoshnevisan et al. \cite{KSXZ2013}
 developed an approximation approach to study the temporal regularity of the solution     $\{u(t, x)\}_{t \ge 0}$ to  the    stochastic fractional heat equation
       \begin{equation}\label{SFHE}
      \frac{\partial  u_{\alpha}(t,x)}{\partial t}= -(-\Delta)^{\frac{\alpha}{2}}u_{\alpha}(t,x) + \sigma(u_{\alpha}(t, x))\dot{W}(t,x),\quad   t>0,\,   x\in\RR,
    \end{equation}
 where $-(-\Delta)^{\frac{\alpha}{2}}, \alpha\in (1,2]$, is the fractional Laplacian on $\mathbb R$ and $\sigma : \RR \, \mapsto \RR$ is a globally Lipschitz continuous function. The  key idea in \cite{KSXZ2013} is to show that  for every fixed $t>0$ and $x \in\mathbb R$, as $\e\downarrow 0$,
            $$
            u_{\alpha}(t+\e, x)-u_{\alpha}(t, x)\approx \frac{1}{\sqrt{\pi(\alpha-1)}}\Gamma\left(\frac{1}{\alpha}\right)^{\frac12}\sigma(u_{\alpha}(t, x)) \left[B^{H_0}(t+\e)-B^{H_0}(t) \right],
                        $$
          in a certain sense, where  $B^{H_0}$ denotes a  fBm with Hurst index $H_0=\frac{\alpha-1}{4}\in \left(0, \frac14\right]$.
            They were able to quantify the size of the approximation error by controlling its moments. Consequently,
            some of the local properties of $t \mapsto u(t, x)$ can be derived from those of the fBm $B^{H_0}$. 
            This type of local linearization was independently developed around the same time in \cite{HP15}, which established nearly sharp error bounds under additional smoothness assumptions on $\sigma$.
 The argument provides a quantitative framework for analyzing the local   structure  of the solutions to parabolic SPDEs, as further explored in \cite{CHKK19, DNP2025, Das2022, HK2017, KKM23}.

    Recently,  Wang and Xiao \cite{WX2024}  extended the approximation approach of \cite{KSXZ2013} to the stochastic fractional heat equation driven by 
    a centered Gaussian field that is white  in time and has the covariance of  a  fBm with Hurst parameter $H\in \left(\frac12,1\right)$ in the space variable.
   In this paper,   we continue this line of research and extend the approximation approach in \cite{KSXZ2013, WX2024} to  SHE \eqref{SHE rough}  in 
     setting of \cite{HHLNT2017} for the rough noise with $H\in \left(\frac14,  \frac12\right)$.  

        Let us first consider the  linear SHE
       \begin{equation}\label{eq SHE linear}
        \frac{\partial  v(t, x)}{\partial t} =\frac{\partial ^2  v(t,x)}{\partial x^2}  +   \dot W(t, x), \ \ \ t>0, x\in \mathbb R,
        \end{equation}
    with the initial condition $v(0,  x)= 0$ for all $x\in \mathbb R.$  The solution  to \eqref{eq SHE linear}  is given by
        \begin{equation}\label{eq mild SHE}
     v(t,x) = \int_0^t \int_{\mathbb R} p_{t-s}(x-y)  W(ds,dy),
    \end{equation}
    where  $p_t(x)$ is the heat kernel, defined as
      \begin{equation}\label{eq heat}
   p_t(x) =\frac{1}{\sqrt{4\pi t}}e^{-\frac{x^2}{4t}}.
     \end{equation}

     Using the idea of the  pinned string process from Mueller and Tribe \cite{MT02}, it can be shown (cf. \cite{TX17, KT2019,   HSWX20}) that,
     for any fixed $x\in \mathbb R$,   there exists a Gaussian process $\{T(t)\}_{t\ge0}$, which is  independent of $ \{v(t, x)\}_{(t,x)\in \mathbb R_+\times \mathbb R}$ and  has a version that is infinitely differentiable on $(0,\infty)$, such that
          \begin{align}\label{eq decom}
   \kappa^{-1}  \left(v(t,x)+T(t) \right) \ \ \ (t\ge0)
       \end{align}
       is a  fBm with Hurst parameter $\frac{H}{2}$,  where
    \begin{equation}\label{eq constant 1}
     \kappa:=\left(  \frac{ \Gamma(2H)}{ \Gamma(H)} \right)^{\frac12}.
     \end{equation}
      See   Lemma \ref{lem decom} below  for details.

  To study   the well-posedness of  SHE \eqref{SHE rough},   Hu et al. \cite{HHLNT2017}   proposed  the following condition:
    \begin{condition}\label{cond A}
     \begin{itemize}
     \item[(A1)] For some $\beta_0> \frac{1}{2}-H$ and some   $p_0 > \max \left(\frac 6 {4H-1}, \frac{1}{\beta_0+H-{1}/{2}}\right)$, the initial value $u_0$ is in $L^{p_0}(\RR) \cap L^\infty(\RR)$ and
     \begin{equation}\label{eq u0}
     \begin{split}
         &\sup_{x\in\RR}  \int_{\RR}|u_0(x)-u_0(x+h)|^2 |h|^{-1-2\beta_0}dh  \\
         & \qquad + \int_{\RR}\|u_0(\cdot)-u_0(\cdot+h)\|^2_{L^{p_0}(\RR)}|h|^{2H-2}dh  < \infty.
     \end{split}
     \end{equation}
     \item[(A2)] $\sigma$ is Lipschitz    and differentiable, the derivative of $\sigma$ is Lipschitz, and $\sigma(0)=0$.
     \end{itemize}
    \end{condition}
 Let 
 $$
 p_t*u_0(x):=\int_{\mathbb R} p_t(x-y)u_0(y)dy.
 $$
  For any $\e>0$ and any random field $\{\eta(t, x)\}_{t\ge0, x\in \mathbb R}$, denote
        \begin{align}\label{eq differ}
         (\mathcal D_{\varepsilon} \eta)(t,x):=\eta(t+\varepsilon,  x)-\eta(t, x), \ \ \ \ \ \ \ \ \ t\ge0, x\in \mathbb R.
         \end{align}
Let
 \begin{align}\label{eq vartheta}
 \vartheta_0:=\frac{1}{2} \left( H \wedge \beta_0\right),
 \end{align}
 and
  \begin{align}\label{eq interval}
  \mathcal I:=\left(0, \frac{(2-H)\vartheta_0}{2(2+\vartheta_0)}\right] \cap \left(0, \frac{(2-H)(1-2H)}{2(5-2H)}\right).
  \end{align}
 \begin{theorem}\label{thm main} Assume that Condition \ref{cond A} holds. For every $p\ge 1$ and $\delta \in \mathcal I$,  there exists a constant   $c_{1,2}\in (0, \infty)$
  such that  for  all  $\varepsilon\in(0,1)$,  $t\in [0,T]$, and $x\in \mathbb R$,
         \begin{equation}\label{eq main}
         \begin{split}
         \Big\| (\mathcal D_{\varepsilon} u)(t,x) &-\left[ p_{t+\e}(\cdot)-p_t(\cdot)\right]*u_0(x)
          -\sigma(u(t, x))(\mathcal D_{\varepsilon} v)(t,x)   \Big\|_{L^{p}(\Omega)}
        \le   c_{1,2}\,   \varepsilon^{\frac{H}{2}+\delta}.
        \end{split}
         \end{equation}
               \end{theorem}

    By Theorem \ref{thm main}  and using similar arguments from \cite{FKM2015, WX2024}, we obtain the following results on the temporal regularity of
    $\{u(t,x)\}_{t\ge0}$ for any fixed $x\in \mathbb R$.

      \begin{proposition}\label{coro t LIL}  Assume that Condition \ref{cond A} holds.
        Choose and fix $t>0$ and $x\in \mathbb R$.   Then, with probability one,
        \begin{itemize}
        \item[(a)] (Khinchin's law of the iterated logarithm)
        \begin{equation}\label{Eq:LIL}
        \begin{split}
        \lim_{\e \to 0} \sup_{0\le r\le \e} \frac{|u(t+r, x)-u(t, x)|}{r^{H/2}\sqrt{2 \log\log(1/r)}}= \kappa  |\sigma(u(t, x))|.
        \end{split}
        \end{equation}
            \item[(b)]  (Chung's law  of the iterated logarithm)
        \begin{equation}\label{Eq: CLIL}
        \begin{split}
        \liminf_{\e \to 0}\sup_{0\le r\le \e}\frac{|u(t+r, x)-u(t, x)|}{(\e/  \log\log(1/\e))^{H/2}}=  \kappa  \lambda_{H}^{H/2} |\sigma(u(t, x))|,
        \end{split}
        \end{equation}
        where $\lambda_{H}$ is the small ball constant of a  fBm with index $\frac{H}{2}$(see, e.g., \cite[Theorem 6.9]{LS01}).
      \end{itemize}
             In addition,     if  $\beta_0>H$ in Condition \ref{cond A}, then the above two results  also  hold at $t=0$ with $\kappa$  
             in \eqref{Eq:LIL} being replaced by  
             \begin{align}\label{eq con kappa1}
             \widetilde\kappa= \left( \frac{ \Gamma(2H) }{2^{1-H} \Gamma(H)}\right)^{1/2}.
             \end{align}
                    \end{proposition}

 We remark that when $W$ is space-time white noise (i.e., $H=\frac{1}{2}$), Foondun et al. \cite{FKM2015} developed an approximation 
 approach to investigate local and variational properties of solutions $\{u_{\alpha}(t, x)\}_{x \in \mathbb{R}}$ to the stochastic fractional 
 heat equation \eqref{SFHE} for fixed $t>0$. Their key insight establishes that as $\varepsilon \downarrow 0$,
\begin{align*}
& u_{\alpha}(t, x+\varepsilon) - u_{\alpha}(t, x) \\
\approx & \left(2\Gamma(\alpha)\left|\cos\left(\frac{\alpha\pi}{2}\right)\right|\right)^{-\frac{1}{2}} \sigma(u_{\alpha}(t, x)) 
\left[B^{(\alpha-1)/2}(x+\varepsilon) - B^{(\alpha-1)/2}(x)\right],
\end{align*}
where the approximation holds in a suitable sense. By carefully controlling moment estimates of the approximation error, they deduced 
local properties of the spatial process $x \mapsto u(t,x)$ from the corresponding properties of fractional Brownian motion.
Recently, Wang \cite{Wang2024} extended  results in \cite{FKM2015}  to the stochastic fractional heat equation \eqref{SFHE} driven 
by a centered Gaussian field that is white in time and exhibits fractional Brownian motion covariance structure with Hurst parameter 
$H \in (\frac{1}{2}, 1)$ in space. Recently, Hu and Lee  \cite{HL25} studied local  spatio-temporal regularities of the  solutions to  
nonlinear parabolic SPDEs driven by the space-time white noise on  a bounded interval by using  the method of linearization.
An interesting problem is to adapt those approximation methods to study the spatial or spatio-temporal regularities 
of the solution $\{u(t,x)\}_{t\ge0, x\in\mathbb{R}}$ to the rough stochastic heat equation \eqref{SHE rough}.   
 
We organize the remainder of this paper as follows. In Section 2, we present preliminary results on stochastic integration and 
properties of solutions to Equation \eqref{SHE rough}, adapted from \cite{HHLNT2017}. In Section \ref{sec proof}, we introduce 
a key lemma to control approximation errors and adapt techniques from \cite{KSXZ2013, WX2024} to prove Theorem \ref{thm main}; 
we defer the detailed proof of this lemma to Section 4. Next, we establish Proposition \ref{coro t LIL} by combining our strong 
approximation error estimates with Khinchin's and Chung's laws of the iterated logarithm for the linear SHE  in Section 5. In Section 
6, we analyze temporal quadratic  variation  for the nonlinear SHE. Finally, we compile technical auxiliary results in the Appendix.

    \section{Preliminaries}\label{lemmas}
    \subsection{Covariance structure and stochastic integration}
    Recall some notations from \cite{HHLNT2017}.   Let $ \cD(\RR) $ denote the space of real-valued, infinitely differentiable functions with compact support on $\mathbb{R}$.
    The Fourier transform of a function $f\in\cD(\RR)$ is defined by
     $$
      \cF f(\xi):=\int_{\RR} e^{-i\xi x}f(x) dx.
    $$

     Let   $(\Omega,\cF,\mathbb P)$ be a complete probability space and let    $\mathcal{D}(\mathbb{R}_{+}\times \mathbb{R})$  denote  the space of real-valued, infinitely differentiable functions with compact support on $\mathbb{R}_{+}\times \mathbb{R}$.
    The noise $\dot W$ is  a zero-mean Gaussian family $\{W(\phi), \phi\in\cD(\RR_{+}\times \RR)\}$  with   the covariance structure given by
     \begin{equation}\label{CovStru}
       \EE\left[ W(\phi)W(\psi)\right]=c_{1,1}\int_{\RR_{+}\times \RR} \cF \phi(s,\xi) \overline{\cF \psi(s,\xi)}\cdot |\xi|^{1-2H} ds d\xi,
     \end{equation}
     where $H\in \left(\frac14 , \frac 12\right)$, $c_{1,1}$ is given in \eqref{e.c1},
     and $\cF \phi(s,\xi)$ is the Fourier transform of the function $\phi(s,x)$ with respect to the spatial variable $x$.

     Let $\mathcal H$  be  the Hilbert space obtained by completing $\mathcal D(\RR) $ with respect to the following scalar  product:      	
     \begin{equation}\label{eq H product}
    \begin{split}
         \langle\phi,\psi\rangle_{\mathcal H}
         :=&\,c_{1,1}\int_{\RR} \cF \phi(\xi) \overline{\cF \psi(\xi)}  \cdot |\xi|^{1-2H}d\xi\\
         =&\, H\left(\frac 12-H\right) \int_{\RR^2}\left[\phi(x+y)-\phi(x)\right]\cdot \left[\psi(x+y)-\psi(x)\right]\cdot |y|^{2H-2} dxdy,
            \end{split}
       \end{equation}
       for any $\phi,\psi\in \mathcal D(\mathbb R)$. 
  The second equality in \eqref{eq H product} holds by \cite[Proposition 2.8]{BJQ2015}.

      For any $t\ge0$, let $\cF_t$ be the filtration generated by $W$, that is,
    $$
     \cF_t:=\sigma\big\{W(\phi): \phi\in \cD([0,t]\times\RR)\big\},
    $$
    where $\cD([0,t]\times\RR)$ is  the space of real-valued infinitely differentiable functions on $[0,t]\times\RR$ with compact support.

    \begin{definition}\cite[Definition 2.2]{HHLNT2017}\label{ElemP}
    An elementary process $g$ is a process given by
      $$
        g(t,x)=\sum_{i=1}^{n}\sum_{j=1}^{m} X_{i,j}\1_{(a_i,b_i]}(t)\1_{(h_j,l_j]}(x),
       $$
       where $n$ and $m$ are finite positive integers, $0\leq a_1<b_1<\cdots<a_n<b_n<\infty$, $h_j<l_j$, and $X_{i,j}$ are $\cF_{a_i}$-measurable
       random variables for $i=1,\dots,n$, $j=1,\dots,m$. The stochastic integral of  an elementary  process $g$ with respect to $W$ is defined as
       \begin{equation*}
         \begin{split}
           \int_{\RR_{+}}\int_{\RR} g(t,x)W(dt,dx) 
          := \, \sum_{i=1}^{n}\sum_{j=1}^{m}  X_{i,j}\left[W(b_i,l_j)-W(a_i,l_j)-W(b_i,h_j)+W(a_i,h_j)\right].
         \end{split}
       \end{equation*}
     \end{definition}

    Hu et al. \cite[Proposition 2.3]{HHLNT2017}  extended the notion of the stochastic integral with respect to $W$ to a broader class of adapted processes
     in the following way:
    \begin{proposition}(\cite[Proposition 2.3]{HHLNT2017})\label{prop 2.3}
       Let $\Lambda_{H}$ be the space of predictable processes $g$ defined on $\RR_{+}\times\RR$ such that  almost surely $g \in L^2(\RR_+; \mathcal{H})$  and
       $\EE[\int_{\RR_+}\|g(s)\|_{\mathcal H}^2ds]<\infty$.  Then,  the following items hold:
       \begin{itemize}
           \item[(i)]
      The space of the elementary processes
        defined in Definition \ref{ElemP} is dense in $\Lambda_{H}$;
          \item[(ii)]
     For any  $g\in\Lambda_{H}$, the stochastic integral $\int_{\RR_{+}}\int_{\RR} g(s,x)W(ds,dx)$ is defined  as the $L^2(\Omega)$-limit of  Riemann sums along
    elementary processes approximating $g$ in $\Lambda_H$, and
        \begin{equation}\label{Isometry}
         \EE\lk\lc\int_{\RR_{+}}\int_{\RR} g(s,x)W(ds,dx)\rc^2\rk=\EE \left[\int_{\RR_+}\|g(s)\|_{\mathcal H}^2ds\right].
        \end{equation}
    \end{itemize}
    \end{proposition}

      Let $(B,\| \cdot \|_B)$ be a Banach space  with the norm $\| \cdot \|_B$.
     For  any  function $f:\RR\rightarrow B$,  we define the norm $\cN_{\frac{1}{2}-H}^{B}f(x)$ as follows:
     \begin{equation}\label{NBNorm}
       \cN_{\frac{1}{2}-H}^{B}f(x):=\lt(\int_{\RR}\|f(x+h)-f(x)\|_B^2\cdot |h|^{2H-2}dh\rt)^{\frac 12},
     \end{equation}
     provided  the integral is convergent.
     When $B=\RR$, we abbreviate this notation  as  $\cN_{\frac{1}{2}-H}f$.
    When $B=L^p(\Omega)$, we  denote $ \cN_{\frac{1}{2}-H}^{B}$ by $\cN_{\frac{1}{2}-H,\,p}$, that is,
     \begin{equation}\label{NpNorm}
       \cN_{\frac{1}{2}-H,\,p}f( x):=\lt(\int_{\RR}\|f(x+h)-f(x)\|^2_{L^p(\Omega)} \cdot |h|^{2H-2}dh\rt)^{\frac 12}.
     \end{equation}
 \subsection{Stochastic heat equation}
  Recall that the predictable $\sigma$-field  $\mathcal P$  is generated by the family of random variables $\{ X(\omega)\cdot {\bf 1}_{(s,t]}(u)\cdot {\bf 1}_A(x)\}$ where $
  t>s\ge 0$,  $A\in \mathcal B(\mathbb R)$,  and $X\in \mathcal F_s$.   A random field $u: \Omega\times \mathbb R_+\times \mathbb R\rightarrow \mathbb R $ is called
  predictable if $u$ is $\mathcal P$-measurable.

    \begin{definition}\cite[Definition 2.4]{HHLNT2017} \label{def-sol-sigma}
    Let $u=\{u(t,x)\}_{t\ge0, \, x \in \mathbb{R}}$ be a real-valued predictable stochastic field such that, for any fixed $t\ge 0$ and $x\in\RR$,
    $$\left\{p_{t-s}(x-y)\sigma(u(s,y)) \1_{[0,t]}(s)\right\}_{0 \leq s \leq t, y \in \mathbb{R} }$$
    is an element of $\Lambda_{H}$, where $p_t(x)$ is given by \eqref{eq heat}. We say that $u$ is a mild solution to  SHE \eqref{SHE rough} if,  for all $t \in [0,T]$
    and $x\in \mathbb{R}$,
    \begin{equation}\label{eq  SHE solution}
    u(t, x)=p_t*u_0(x) + \int_0^t \int_{\mathbb{R}}p_{t-s}(x-y)\sigma(u(s,y)) W(ds,dy) \quad a.s.,
    \end{equation}
    where the stochastic integral is understood in the sense of Proposition \ref{prop 2.3}.
    \end{definition}

  \begin{remark}\label{rem u0}
        Note that Condition \ref{cond A} automatically implies  that \eqref{eq u0} holds for all $p\ge p_0$.
    Indeed, by Proposition A.1 in \cite{HHLNT2017},      we know that  Condition \ref{cond A} implies that the initial function  $u_0$ is H\"older continuous with order $\beta_0$.
    Consequently,  by H\"older's inequality,  we obtain the following for any $p>p_0$, 
    \begin{equation}\label{eq u0p}
    \begin{split}
      &\,  \int_{|h|\le 1} \left(\int_{\mathbb R}|u_0(x)-u_0(x+h)|^p dx\right)^{\frac{2}{p} } |h|^{2H-2}dh\\
     \le &\,  C   \int_{|h|\le 1} \left(\int_{\mathbb R}|u_0(x)-u_0(x+h)|^{p_0} dx\right)^{\frac{2}{p} } |h|^{\frac{2\beta_0(p-p_0)}{p} + 2H-2}dh\\
           \le &\,  C  \left(  \int_{|h|\le 1} \left(\int_{\mathbb R}|u_0(x)-u_0(x+h)|^{p_0} dx\right)^{\frac{2}{p_0 } } |h|^{  2H-2}dh \right)^{\frac{p_0}{p}}\\
           &\,\,\, \times          \left(  \int_{|h|\le 1} |h|^{2\beta_0+  2H-2}dh \right)^{1-\frac{p_0}{p}}.
            \end{split}
            \end{equation}
  Both integrals in the last term are finite due to \eqref{eq u0} and the fact that $\beta_0>\frac12-H$ as stated in Condition \ref{cond A}. Therefore,
  by \eqref{eq u0p} and the fact that  $u_0\in   L^{\infty}(\mathbb R)$  according to  Condition \ref{cond A}, we know that \eqref{eq u0} holds for all $p\ge p_0$.
     \end{remark}

The following result about the solution of Equation \eqref{SHE rough} follows from  \cite[Theorems 4.5 and 4.7]{HHLNT2017} and Remark \ref{rem u0} above.   
          \begin{theorem} \label{them existence}
     Under Condition \ref{cond A},   Equation \eqref{SHE rough} admits a unique mild solution satisfying the following moment bounds: for any $p \ge  p_0$ and $T>0$,
        \begin{equation}\label{eq bound}
    \sup_{x\in \mathbb R, t\in [0,T]}\|u(t,x)\|_{L^p(\Omega)}+ \sup_{x\in \mathbb R,t\in [0,T] }\mathcal N_{1/2-H, p} u(t,x)<\infty,
    \end{equation}
and  for all $s, t \in [0,T], x,y \in \RR$,
     \begin{equation}\label{eq Holder}
    \|u(t,x)-u(s,y)\|_{L^p(\Omega)} \leq c_{2,1} \left(|t-s|^{\vartheta_0}+ |x-y|^{2\vartheta_0}\right),
    \end{equation}
     where $\vartheta_0$ is defined by \eqref{eq vartheta}.
       \end{theorem}

      \begin{proposition}\label{HBDG}   $($\cite[Proposition 3.2]{HHLNT2017}$)$
       Let $W$ be the Gaussian noise with the covariance given by \eqref{CovStru}, and let $f\in\Lambda_H$  be a predictable random field. Then, for any $p\geq2$, 
      \begin{equation}\label{eq BDG}
         \begin{split}
            \lt\|\int_{0}^{t}\int_{\RR}f(s,y) W(ds,dy)\rt\|_{L^p(\Omega)}
         \leq \sqrt{4p}c_{2,2}\lt(\int_{0}^{t}\int_{\RR}\lk\cN_{\frac 12-H,\,p}f(s,y)\rk^2dyds\rt)^{\frac 12},
         \end{split}
       \end{equation}
       where $c_{2,2}$ is a constant depending only on $H$, and  $\cN_{\frac 12-H,\,p}f(s,y)$ denotes the application of $\cN_{\frac 12-H,\,p} $  to the spatial variable $y$.
     \end{proposition}

We note that Chen and Dalang \cite[Proposition 3.1]{CD15} and Chen and Kim \cite[Proposition 2.2]{CK2019} provided easily verifiable conditions to check
the predictability of random fields. In particular,   the Burkholder-Davis-Gundy (BDG, for short) inequality in Proposition \ref{HBDG}  is valid for the solution $u=\{u(t, x)\}_{t\ge0, x\in \mathbb R} $ to
\eqref{SHE rough}, since $u$ is jointly measurable with respect to $\mathcal B(\mathbb R_+\times \mathbb R)\times \mathcal F$;  $u$ is adapted, i.e., for all
$(s, x)\in (0,\infty)\times \mathbb R$, $u(s,x)$  is $\mathcal F_s$-measurable; and $u$ satisfies the moment bounds in \eqref{eq bound}.

  \section{Proofs of  main results}\label{sec proof}

   In this section, we will employ a modified argument from  \cite{KSXZ2013, WX2024} to estimate the temporal increment. 
       For any fixed $x\in\mathbb R$,   the increment of  the process $t\mapsto u(t, x)$ can be written as
      \begin{align}\label{eq Diff X}
    u(t+\varepsilon, x)-u(t, x)  =\, \left[ p_{t+\e}(\cdot)-p_t(\cdot)\right]*u_0(x)+  \mathcal J_1+\mathcal J_{2, \theta}+\mathcal J_{3, \theta},
      \end{align}
      where
      \begin{align}
            \mathcal J_1:=&\, \int_t^{t+\varepsilon} \int_{\mathbb R} \,p_{t+\varepsilon-s}(x-y) \sigma\left(u(s, y)\right)W(ds,dy); \label{eq J1} \\
        \mathcal J_{2, \theta}:=&\, \int_0^{t-\varepsilon^{\theta}}\int_{\mathbb R}\left[p_{t+\varepsilon-s}(x-y)-p_{t-s}(x-y)\right]\sigma(u(s, y))W(ds,dy);   \label{eq J2} \\
          \mathcal J_{3, \theta}:=&\, \int_{t-\varepsilon^{\theta}}^t\int_{\mathbb R}\left[p_{t+\varepsilon-s}(x-y)-p_{t-s}(x-y)\right]\sigma(u(s, y))W(ds,dy).  \label{eq J3}
                   \end{align}
           Here,  the parameter $\theta\in (0,1)$ is fixed, and its optimal value will be specified later.
    For any $t>0$ and for any sufficiently small $\varepsilon$,   we define $t(\varepsilon):=t-\varepsilon^{\theta}$ for convenience.

 \subsection{An important lemma}
    We outline the idea of approximating the three terms in \eqref{eq Diff X}, inspired by \cite{KSXZ2013, WX2024}:
             \begin{itemize}
             \item Firstly,  since  $p_{t+\e-s}(x-y)$ behaves like  a point mass  when $s\in (t, t+\varepsilon)$ and $\varepsilon$ is small,
             and $u$ is H\"older continuous, we expect that
  $\mathcal J_{1}\approx  \widetilde{\mathcal J}_{1}$, where
          \begin{align}\label{eq appr J1}
          \widetilde{ \mathcal J}_{1}:= \sigma\left(u(t, x)\right)\int_t^{t+\varepsilon} \int_{\mathbb R} \,p_{t+\varepsilon-s}(x-y)  W(ds,dy).
       \end{align}
       \item
     Secondly,  we will show that the quantity $\mathcal J_{2, \theta}$ becomes small as long as $\theta\in (0,1)$ is chosen appropriately.

         \item Thirdly,  since  $p_{t+\e-s}(x-y)$ and $p_{t-s}(x-y)$ both act as    point masses  when $s\in (t(\e), t)$,  we expect that
         $\mathcal J_{3, \theta}\approx \mathcal J_{3, \theta}'$, where
          \begin{align}\label{eq appr 2}
      \mathcal J_{3, \theta}':=&\, \sigma\left(u\left(t(\varepsilon),x\right)\right) \int_{t(\varepsilon)}^t\int_{\mathbb R}\left[p_{t+\varepsilon-s}(x-y)-p_{t-s}(x-y)\right] W(ds,dy).
       \end{align}
       \item  Finally, by the H\"older continuity of $u$, we  expect that $\mathcal J_{3, \theta}' \approx  \widetilde{\mathcal J}_{3,\theta}$,
       where
  \begin{equation}\label{eq appr 3}
           \widetilde{\mathcal J}_{3, \theta}:= \, \sigma(u(t, x))\int_{t(\varepsilon)}^t\int_{\mathbb R}\left[p_{t+\varepsilon-s}(x-y)-p_{t-s}(x-y)\right]W(ds,dy).
           \end{equation}
         \end{itemize}
         More precisely, we quantify the sizes of the approximation errors by controlling their moments, as stated in the following lemma.
  \begin{lemma}\label{lem approximation}   For any  $T>0$ and  $p\ge p_0$, the following statements hold:
 \begin{itemize}
 \item[(a)] For any $\vartheta\in (0,  \vartheta_0]\cap \left(0,  \frac12-H\right)$,  there exists a finite positive constant  $c_{3,1}$,  independent  of  $\varepsilon\in (0,1)$,
 such that
         \begin{equation}\label{eq J2 diff1}
      \sup_{t\in [0,T]} \sup_{x\in \mathbb R}  \left\|\mathcal J_1-\widetilde{ \mathcal J}_{1}\right\|_{L^{p}(\Omega)} ^2 \le c_{3,1} \varepsilon^{ H+\vartheta}.
      \end{equation}

 \item[(b)]
  There exists a finite positive  constant $c_{3,2}$, independent of
      $\varepsilon\in (0,1)$, such that
      \begin{equation}\label{eq J1 1}
     \sup_{t\in [0,T]}  \sup_{x\in \mathbb R}  \left\|\mathcal J_{2, \theta} \right\|_{L^{p}(\Omega)}^2  \le  c_{3,2}\varepsilon^{[H+(2-H)(1-\theta)]\wedge \frac12}.
      \end{equation}

 \item[(c)] For any $\vartheta\in (0,  \vartheta_0]\cap \left(0,  \frac12-H\right)$, there exists a
         finite positive constant $c_{3,3}$ such that  for all  $\varepsilon,\theta\in (0,1)$,
              \begin{equation}\label{eq J 3 est}
               \sup_{t\in [0,T]} \sup_{x\in \mathbb R}\left\|\mathcal J_{3, \theta}-\mathcal J_{3, \theta}' \right\|_{L^{p}(\Omega)}^2 \le c_{3,3} \varepsilon^{ (H+\vartheta )\theta}.
                  \end{equation}

 \item[(d)] There exists a finite positive constant $c_{3,4}$, such that for all  $\varepsilon,\theta\in (0,1)$,
              \begin{align}\label{eq J14 est}
               \sup_{t\in [0,T]}\sup_{x\in \mathbb R} \left\| \mathcal J_{3, \theta}' - \widetilde{\mathcal J}_{3,\theta}\right\|_{L^{p}(\Omega)}^2
               \le c_{3,4}\varepsilon^{(H +2\vartheta_0)\theta}.
              \end{align}
  \end{itemize}
        \end{lemma}

 The proof  of  Lemma \ref{lem approximation}  is deferred to  Section \ref{sec lemma}.   Let us first use it to prove Theorem \ref{thm main}.

  \subsection{Proof of Theorem \ref{thm main}}\label{sec proof thm main}
\begin{proof}[Proof of   Theorem \ref{thm main}] The proof is inspired by  \cite[Proposition 4.6]{KSXZ2013}  and \cite[Theorem 1.2]{WX2024}. For any
$\vartheta\in (0,  \vartheta_0]\cap \left(0,  \frac12-H\right)$,  let
          \begin{align}\label{eq const a}
       \theta:=\frac{2}{2+\vartheta}.
          \end{align}
          For this particular choice of $\theta$, we have
          \begin{align}\label{eq const a2}
          H+(2-H)(1-\theta)=  (H+\vartheta )\theta=:\mathcal G_{\vartheta}.
          \end{align}
          Notice that $\mathcal G_{\vartheta }=H+  \frac{ 2-H }{2+\vartheta} \vartheta \in \left(H, \frac12\right)$.

            By the Minkowski and Jensen inequalities,   parts (c) and (d) of   Lemma \ref{lem approximation}   imply that for any $p\ge 1, \eta\in \left(H,  \mathcal G_{\vartheta}\right]$,     we have
                         \begin{align}\label{J1 diff 1}
                  \left\|  {\mathcal J_3-\widetilde{\mathcal J}_{3, \theta}}\right\|_{L^p(\Omega)}\le c_{3,5} \e^{\eta},
                  \end{align}
             for some  positive finite constant $c_{3,5}$, which does not depend on $\e\in(0,1)$,    $t\in [0,T]$,  or $x\in \mathbb R$.

For any  interval $Q\subset [0,T]$,  let
$$\Lambda(Q):= \int_{Q\times \mathbb R}\left[p_{t+\e-s}(x-y)-p_{t-s}(x-y)\right] W(ds,dy).$$
 By     \eqref{eq appr 3}, we have
   $$ {\widetilde{\mathcal J}_{3, \theta}}  =\sigma(u(t, x)) \Lambda\left(\left[t-\e^{\theta},t\right]\right).$$
     Applying   \eqref{eq bound} and  the Cauchy-Schwarz inequality,
      we obtain  the following estimate
      \begin{align}\label{J1 diff 2}
  \left\|  {\widetilde{\mathcal J}_{3, \theta}}- \sigma(u(t, x))  \Lambda([0, t])\right\|_{L^p(\Omega)} \le c_{3,6}
  \left\| \Lambda\left(\left[0,t-\e^{\theta}\right]\right)\right\|_{L^{2p}(\Omega)},
      \end{align}
      where $c_{3,6}>0$ is a constant.   Since $\Lambda\left(\left[0,t-\e^{\theta}\right]\right)$ is the same as the quantity $\mathcal J_{3, \theta}$ when
      $\sigma\equiv1$,  we can apply  part (b) of Lemma \ref{lem approximation}   to the linear equation \eqref{eq SHE linear}. This gives  the estimate for any
      $\eta\in (H,  \mathcal G_{\vartheta}]$,
          \begin{align}\label{J1 diff 3}
\left\| \Lambda\left(\left[0,t-\e^{\theta}\right]\right)\right\|_{L^{2p}(\Omega)} \le c_{3,2}  \e^{\eta}.
      \end{align}
   Combining  \eqref{J1 diff 1}-\eqref{J1 diff 3}, we obtain  that  every $p\ge 1$ and $\eta \in (H,  \mathcal G_{\vartheta}]$, there exists a  finite  positive constant
   $c_{3,7}$, which does not depend on   $t\in [0,T]$,  $\varepsilon \in (0,1)$, or $x\in \mathbb R $, such that
      \begin{align*}
    \left\| \mathcal J_{3, \theta}-\sigma(u(t, x)) \int_{0}^t\int_{\mathbb R}\left[p_{t+\e-s}(x-y)-p_{t-s}(x-y)\right] W(ds,dy) \right\|_{L^p(\Omega)}
    \le c_{3,7} \e^{\eta}.
                  \end{align*}
    This, together with \eqref{eq J2 diff1} and \eqref{eq J1 1}, implies that \eqref{eq main} holds for every $\eta\in (H,  \mathcal G_{\vartheta}]$.
 The proof is complete.
   \end{proof}

 \subsection{Remarks about the initial value}\label{sec rem initial}
                If $\beta_0 > H$ in Condition \ref{cond A},    then for any $\delta\in \mathcal I$ defined by \eqref{eq interval} with $\vartheta_0=H/2$, there exists  a constant    $c_{3, 8}>0$
       such that for all $\varepsilon\in(0,1)$,
         \begin{align}\label{eq main 2}
      \sup_{t\in[0,T]}\sup_{x\in \mathbb R}   \Big\| (\mathcal D_{\varepsilon} u)(t,x)           -\sigma(u(t, x))(\mathcal D_{\varepsilon} v)(t,x)   \Big\|_{L^{p}(\Omega)}
        \le   c_{3,8}\,   \varepsilon^{\frac{H}{2}+\delta}.
         \end{align}
        Indeed, since the initial condition $u_0$ is H\"older continuous with order $\beta_0$ (see Remark \ref{rem u0}),
      by \cite[Proposition 2.6]{KS2023},   there exists a constant $c_{3,9}>0$ such that
      \begin{align}\label{eq J00}
        \left|\big[ p_{t+\e}(\cdot)-p_t(\cdot)\big]*u_0(x)  \right| \le c_{3, 9}\e^{\frac{\beta_0}{2}}.
        \end{align}
     Here, the   constant  $c_{3,9}$  is independent  of  $\e\in (0,1)$, $t\in [0,T]$,   and $x\in \mathbb R$.  Combining \eqref{eq main} and \eqref{eq J00} yields \eqref{eq main 2}.

    Under Condition \ref{cond A},
     for any fixed $t>0$, and   $x, y\in \mathbb R$, we have
     \begin{equation*}
    \left|p_t*u_0(x)-p_t*u_0(y)\right|
    \le \,  \int_{\mathbb R}      \left| p_t(x-z)-p_t(y-z) \right| dz \cdot \|u_0\|_{L^{\infty}(\mathbb R)}.
      \end{equation*}
      By \cite[Lemma 2.1]{HRZ2021}, we obtain
     \begin{equation*}
     \begin{split}
    \left|p_t*u_0(x)-p_t*u_0(y)\right|
    \le  \,      c_{3,10}   \frac{|x-y|^{\frac{1}{2}}}{t^{\frac{1}{2}}}    \cdot \|u_0\|_{L^{\infty}(\mathbb R)},
    \end{split}
      \end{equation*}
   where  $c_{3,10} >0$ is a constant.   Hence,
      by \cite[Proposition 2.6]{KS2023},   there exists a constant $c_{3,11}>0$   such that
        \begin{align}\label{eq J01}
        \left|\big[ p_{t+\e}(\cdot)-p_t(\cdot)\big]*u_0(x)  \right| \le c_{3,11}\e^{\frac{1}{4}},
        \end{align}
    for all $t\ge T>0, \e\in(0,1)$, and $x  \in \RR$. Putting    \eqref{eq main} and  \eqref{eq J01} together,  we obtain  that for all   $\e>0$ and $T_2>T_1>0$,
        \begin{equation*}
        \begin{split}
      \sup_{t\in [T_1, T_2]}\sup_{x\in \mathbb R}  \Big\| \big[u(t+\e, x)-u(t,x)\big]-\sigma(u(t, x))\big[v(t+\e, x)-v(t,x)\big]  \Big\|_{L^p(\Omega)} 
      \le  \,  c_{3,12}\varepsilon^{\frac{H}{2}+\delta},
      \end{split}
         \end{equation*}
     for some constant   $c_{3,12}>0$.

    Using Parts (a)-(c) of Lemma \ref{lem approximation}  as  in the proof of Theorem \ref{thm main} and using \eqref{eq J01},   we 
    derive the following result.
    \begin{corollary}\label{corollary a-c} Assume that Condition \ref{cond A} holds with $\beta_0 > H$.  For any
$\vartheta\in (0,  \vartheta_0]\cap \left(0,  \frac12-H\right)$,  let $\theta$ be the constant given by
        \eqref{eq const a}. Then,  there exists a constant   $c_{3,13}\in (0, \infty)$
   such that  for  all  $\varepsilon\in(0,1)$,  $t\in [0,T]$, and $x\in \mathbb R$,
          \begin{equation}\label{eq main 21}
          \begin{split}
          \left\| (\mathcal D_{\varepsilon} u)(t,x)   -\sigma\left(u\left(t-\varepsilon^{\theta}, x\right)\right)\left(\mathcal D_{\varepsilon} v\right)(t,x)   \right\|_{L^{p}(\Omega)} 
         \le   c_{3, 13}\,   \varepsilon^{\frac{H}{2}+ \frac{ (2-H) \vartheta }{2(2+\vartheta)} }.
         \end{split}
          \end{equation} 
                             \end{corollary}

    \section{Proof of Lemma \ref{lem approximation}}\label{sec lemma}
         \subsection{Proof of (a) in Lemma \ref{lem approximation}}
    \begin{proof}[Proof of (a) in Lemma \ref{lem approximation}]
     Notice that  $\mathcal J_1-\widetilde{\mathcal J}_1$ has the form
    $$
     \int_t^{t+\varepsilon} \int_{\mathbb R} g_{t, x, \varepsilon}(s,y) W(ds,dy),
    $$
    with
    $$
       g_{t, x, \varepsilon}(s,y):=  p_{t+\varepsilon-s}( x-y) \left[\sigma\big(u(s, y)\big)- \sigma\big(u(t, x)\big)\right].
    $$
  Since  $g_{t, x, \varepsilon}$ is predictable, by   BDG's inequality  \eqref{eq BDG},  we have
     \begin{equation*}
    \begin{split}
     \left\|\mathcal J_1-\widetilde{\mathcal J}_1 \right\|_{L^{p}(\Omega)}^2
    \le  \,  4p c_{2,2}^2 \int_t^{t+\varepsilon}ds \int_{\mathbb R}dy  \int_{\mathbb R}dz   \left\|   g_{t, x, \varepsilon}(s,y) -  g_{t, x, \varepsilon}(s,z)  \right\|^2_{L^{p}(\Omega)}  |y-z|^{2H-2}.
     \end{split}
    \end{equation*}
    Decomposing the increments  $g_{t, x, \varepsilon}(s,y) -  g_{t, x, \varepsilon}(s,z)$  yields
    \begin{equation} \label{eq A 1}
      \left\|\mathcal J_1-\widetilde{\mathcal J}_1 \right\|_{L^{p}(\Omega)}^2 \le   8p c_{2,2}^2 \left[  \int_t^{t+\varepsilon}  J_{1}(s)ds + \int_t^{t+\varepsilon}  J_{2}(s)ds\right],
    \end{equation}
     where
    \begin{equation*} \label{eq A 11}
    \begin{split}
     J_{1}(s):=&\, \int_{\mathbb R}dy  \int_{\mathbb R}dz   \big|p_{t+\varepsilon-s}(  x-y)- p_{t+\varepsilon-s}(  x-z)  \Big|^2\\
    &  \ \ \  \cdot  |y-z|^{2H-2} \left\|\sigma\big(u(s, y)\big)- \sigma\big(u(t, x)\big)\right\|_{L^{p}(\Omega)}^2,
     \end{split}
    \end{equation*}
    and
     \begin{equation*} \label{eq A 12}
    \begin{split}
     J_{2}(s):= \, \int_{\mathbb R}dy  \int_{\mathbb R}dz p_{t+\varepsilon-s}(x-z)^2
       |y-z|^{2H-2} \left\|\sigma\big(u(s, y)\big)- \sigma\big(u(s, z)\big)\right\|_{L^{p}(\Omega)} ^2.
     \end{split}
    \end{equation*}

    By   \eqref{eq bound} and \eqref{eq Holder}, we have that for any $\vartheta\le \vartheta_0$,
     \begin{equation} \label{eq A0 112}
    \begin{split}
     J_{1}(s) \le &\, c_{4,1}
       \int_{\mathbb R}dy  \int_{\mathbb R}dz   \big|p_{t+\varepsilon-s}(  x-y)- p_{t+\varepsilon-s}( x-z)  \Big|^2\\
    &  \ \ \ \  \cdot  |y-z|^{2H-2} \left[\big(s-t\big)^{\vartheta}+ |y-x|^{2\vartheta}\right]\\
    \le &\,
    c_{4,1}  \big(s-t\big)^{\vartheta}
      \int_{\mathbb R}dy  \int_{\mathbb R}dz   \big|p_{t+\varepsilon-s}( x-y)- p_{t+\varepsilon-s}(  x-z)  \Big|^2     |y-z|^{2H-2}   \\
    &\,+ c_{4,1}   \int_{\mathbb R}dy  \int_{\mathbb R}dz   \big|p_{t+\varepsilon-s}(  x-y)- p_{t+\varepsilon-s}(  x-z)  \Big|^2  |y-z|^{2H-2}  |y-x|^{2\vartheta}\\
    =:&\, c_{4,1} \left[J_{1,1}(s)+ J_{1,2}(s)\right].
     \end{split}
     \end{equation}
    Applying  Lemma \ref{lem int p} with $\beta=\frac12-H$, we obtain the estimate
    $$
    \int_{\mathbb R}dy  \int_{\mathbb R}dz   \big|p_{t+\varepsilon-s}(  x-y)- p_{t+\varepsilon-s}(  x-z)  \Big|^2     |y-z|^{2H-2}= c_{4,2} (t+\varepsilon-s)^{H-1}.
    $$
     This allows us to write the integral over time as
     \begin{equation} \label{eq A 11-10}
    \begin{split}
      \int_t^{t+\varepsilon}   J_{1, 1}(s)ds =\,     c_{4,2}   \int_t^{t+\varepsilon} (s-t)^{\vartheta} (t+\varepsilon-s)^{H-1} ds
    = \,  c_{4,2} \mathbf B (1+\vartheta, H)  \varepsilon^{H+ \vartheta},
      \end{split}
     \end{equation}
    where $\mathbf B(1+\vartheta, H)$ is the beta function.

    By a change of variables and the scaling property  of the heat kernel:
        \begin{align}\label{eq scaling}
    p_t(x)=t^{-\frac12}p_1(t^{-\frac12} x),  \ \ \ \text{for } t>0, x\in \mathbb R,
    \end{align}
          we can write  $J_{1,2}(s)$  as follows: for any  $\vartheta<\frac12-H$,
     \begin{equation*}
    \begin{split}
    J_{1,2}(s) =
    &  \int_{\mathbb R}dy  \int_{\mathbb R}dz   \big|p_{t+\varepsilon-s}(  x-y)- p_{t+\varepsilon-s}(  x-z)  \big|^2  |y-z|^{2H-2}  |y-x|^{2\vartheta}\\
    =&\,  (t+\varepsilon-s)^{H+\vartheta-1}\int_{\mathbb R}dy  \int_{\mathbb R}dz   \big|p_1( y)- p_1(  z)  \Big|^2  |y-z|^{2H-2}  |y|^{2\vartheta}.
    \end{split}
     \end{equation*}
  The last double integral  in the above equation is finite due to  Lemma \ref{lem green 2}.  Therefore,
      \begin{equation} \label{eq A  11-2}
    \int_t^{t+\varepsilon}   J_{1,2}(s)ds \le \,   c_{4,3} \int_t^{t+\varepsilon} (t+\varepsilon-s)^{H+\vartheta-1}   ds  = \frac{c_{4,3}}{H+\vartheta} \varepsilon^{H+\vartheta},
     \end{equation}
     where $c_{4,3}$ is a positive constant.

    Recall  the notation $\mathcal  N_{\frac{1}{2}-H,  p} $ defined by  \eqref{NpNorm}.   There exists a constant $c_{4,4}>0$ such that
    \begin{equation*}
    \begin{split}
     J_{2}(s)  =&\,   \int_{\mathbb R}  p_{t+\varepsilon-s}(  x-z)^2 \left[\mathcal N_{\frac{1}{2}-H,  p} \sigma(u(s,z)) \right]^2 dz\\
     \le &\, \sup_{s\in [0,T], x\in \mathbb R} \left[\mathcal N_{\frac{1}{2}-H,  p} \sigma(u(s,x))\right]^2   \int_{\mathbb R}  p_{t+\varepsilon-s}(x-z)^2 dz \\
    \le &\,    c_{4,4} (t+\varepsilon-s)^{-\frac12},
    \end{split}  \end{equation*}
 where  we have used the fact
  $$
     \sup_{s\in [0,T], x\in \mathbb R} \left[\mathcal N_{\frac{1}{2}-H, p} \sigma(u(s,x))\right]^2\le\, L_{\sigma}^2\sup_{s\in [0,T], x\in \mathbb R} \left[\mathcal N_{\frac{1}{2}-H, p} u(s,x)\right]^2<\infty,
     $$
    with $L_{\sigma}$ being the Lipschitz   constant of $\sigma$.
    Thus,
     \begin{equation} \label{eq A 12-2}
    \begin{split}
    \int_t^{t+\varepsilon}  J_{2}(s)ds  \le  2c_{4,4} \varepsilon^{\frac12}.
    \end{split}  \end{equation}

Combining  \eqref{eq A 1},  \eqref{eq A0 112},  \eqref{eq A 11-10}, \eqref{eq A 11-2}, and \eqref{eq A 12-2}, we obtain  \eqref{eq J2 diff1}, which completes           the proof of part (a) in Lemma  \ref{lem approximation}.
    \end{proof}

    \subsection{Proof of (b) in Lemma  \ref{lem approximation}}
    For any $\varepsilon>0$,  denote
        \begin{align}\label{eq differ2}
              p_{[t,\,  t+\varepsilon]}(x):=p_{t+\varepsilon}(x)-p_t(x), \ \ \ \ \ \ \ \ \ (t\ge 0, x\in \mathbb R).
             \end{align}

     \begin{proof}[Proof of (b) in Lemma  \ref{lem approximation}]
    The   term $\mathcal J_{2, \theta} $ has  the form
    $$
     \int_0^{t(\varepsilon)}  \int_{\mathbb R} f_{t, x, \varepsilon}(s,y) W(ds,dy),
    $$
    with
    $$
    f_{t, x, \varepsilon}(s,y) :=  p_{[t-s, \, t+\varepsilon-s]}(x-y)\sigma\big(u(s, y)\big).
    $$
     Since  $f_{t, x, \varepsilon}$ is predictable, applying  the BDG inequality  \eqref{eq BDG},  we have
     \begin{equation*} \label{eq A 3}
    \begin{split}
    \left\|\mathcal J_{2, \theta} \right\|_{L^{p}(\Omega)}^2
    \le  \,  4pc_{2,2}^2 \int_0^{t(\varepsilon)} ds \int_{\mathbb R}dy  \int_{\mathbb R}dz   \left\|  f_{t, x, \varepsilon}(s,y) - f_{t, x, \varepsilon}(s,z)  \right\|^2_{L^{p}(\Omega)}  |y-z|^{2H-2}.
     \end{split}
    \end{equation*}
    Decomposing    the increments  $f_{t, x, \varepsilon}(s,y) -  f_{t, x, \varepsilon}(s,z)$   yields
    \begin{equation}\label{eq A decomp}
    \left\|\mathcal J_{2, \theta} \right\|_{L^{p}(\Omega)}^2 \le 4p c_{2,2}^2  \left[  \int_0^{t(\varepsilon)}  J_{3}(s)ds + \int_0^{t(\varepsilon)}  J_{4}(s)ds\right],
     \end{equation}
     where
    \begin{equation*} \label{eq A 31}
    \begin{split}
     J_{3}(s):= & \, \int_{\mathbb R}dy  \int_{\mathbb R}dz  \Big| p_{[t-s,\, t+\varepsilon-s]}(x-y)- p_{[t-s,\, t+\varepsilon-s]}( x-z)  \Big|^2
     |y-z|^{2H-2} \left\|\sigma\big(u(s, y)\big) \right\|_{L^{p}(\Omega)} ^2,\\
   J_{4}(s):=&  \, \int_{\mathbb R}dy  \int_{\mathbb R}dz\,  p_{[t-s,\, t+\varepsilon-s]}( x-z)^2
       |y-z|^{2H-2} \left\|\sigma\big(u(s, y)\big)- \sigma\big(u(s, z)\big)\right\|_{L^{p}(\Omega)} ^2.
     \end{split}
    \end{equation*}

    By \eqref{eq bound}, a change of variables,  and   Plancherel's identity,   we have
    \begin{equation} \label{eq A 3101}
    \begin{split}
     &\int_0^{t(\varepsilon)}  J_{3}(s) ds\\
      \le &\, c_{4,5}\int_0^{t(\varepsilon)}ds    \int_{\mathbb R}dy  \int_{\mathbb R}dz  \Big| p_{[t-s,\, t+\varepsilon-s]}(x-y)- p_{[t-s,\, t+\varepsilon-s]}( x-z)  \Big|^2
      |y-z|^{2H-2}\\
     =  &\, c_{4,5}\int_0^{t(\varepsilon)}ds    \int_{\mathbb R}dy  \int_{\mathbb R}dh  \Big| p_{[t-s,\, t+\varepsilon-s]}( y)- p_{[t-s,\, t+\varepsilon-s]}(y+h)  \Big|^2
      |h|^{2H-2} \\
      =&\,  c_{4,5}\int_0^{t(\varepsilon)}ds   \int_{\mathbb R}d\xi   \int_{\mathbb R}dh \, e^{-2(t-s)  \xi^2} \left|e^{-\varepsilon\xi^2} -1\right|^2 \left| e^{-i\xi h}-1\right|^2  |h|^{2H-2}\\
    = &\,  c_{4,5}\int_0^{t(\varepsilon)}ds   \int_{\mathbb R}d\xi    \, e^{-2(t-s)  \xi^2} \left|e^{- \varepsilon\xi^2} -1\right|^2   |\xi|^{1-2H} \cdot \int_{\mathbb R}dh  \left| e^{-i  h}-1\right|^2
     |h|^{2H-2},
     \end{split}
     \end{equation}
     where $c_{4, 5}\in (0,\infty)$ is a constant.
         By the elementary inequality $0\le 1-e^{-x}\le   x$ for all $x\ge0$, we have
   \begin{equation} \label{eq A 3102}
    \begin{split}
  &  \int_0^{t(\varepsilon)}ds   \int_{\mathbb R}d\xi    \, e^{-2(t-s)  \xi^2} \left|e^{- \varepsilon\xi^2} -1\right|^2   |\xi|^{1-2H}\\
     \le&\,  \varepsilon^2\int_0^{t(\varepsilon)}ds \int_{\mathbb R}d\xi     \, e^{-2(t-s)  \xi^2}|\xi|^{5-2H}\\
    = &\, 2^{H-3} \varepsilon^2\int_0^{t(\varepsilon)} (t-s) ^{H-3}ds \cdot \int_{\mathbb{R}}e^{-\xi^2}\xi^{5-2H}d\xi\\
     \le&\, 2^{H-3}  \varepsilon^{H+(2-H)(1-\theta)}\int_{\mathbb{R}}e^{-\xi^2}\xi^{5-2H}d\xi.
   \end{split}
     \end{equation}
    Since the integrals $\int_{\mathbb{R}}e^{-\xi^2}\xi^{5-2H}d\xi$ and  $\int_{\mathbb R}  \left| e^{-i  h}-1\right|^2  |h|^{2H-2} dh$ are  convergent,     \eqref{eq A 3101} and \eqref{eq A 3102} imply that
    \begin{equation}\label{eq A 310}
     \int_0^{t(\varepsilon)}  J_{3}(s) ds \le c_{4,6}\varepsilon^{H+(2-H)(1-\theta)},
    \end{equation}
   for some positive constant $c_{4,6}\in (0,\infty)$.


    Recalling the notation $\mathcal N_{\frac{1}{2}-H, p}$ defined  in  \eqref{NpNorm}, and using  Plancherel's identity along with  \eqref{eq bound}, we  have
    \begin{equation} \label{eq A 320}
    \begin{split}
     \int_0^{t(\varepsilon)}  J_{4}(s) ds =&\, \int_0^{t(\varepsilon)}ds  \int_{\mathbb R} dz  p_{[t-s,\, t+\varepsilon-s]}(x-z)^2  \left[\mathcal N_{\frac{1}{2}-H, p} \sigma(u(s,z) )\right]^2 \\
     \le &\, \sup_{s\in [0,T],\,  x\in \mathbb R}  \left[\mathcal N_{\frac{1}{2}-H, p} \sigma(u(s,x)) \right]^2   \cdot  \int_0^{t(\varepsilon)}ds   \int_{\mathbb R}dz\,  p_{[t-s,\, t+\varepsilon-s]}(x-z)^2    \\
    \le &\,    c_{4,7}   \int_0^{t(\varepsilon)} ds   \int_{\mathbb R}  d\xi  \,   e^{-2(t-s)  \xi^2} \left|e^{-  \varepsilon \xi^2} -1\right|^2  \\
    = &\,  \frac{c_{4,7} }{2} \int_{\mathbb R}  \left[e^{-2\e^{\theta}\xi^2}-e^{-2t\xi^2} \right]   \left|e^{- \varepsilon   \xi^2}-1  \right|^2 |\xi|^{-2}d\xi\\
    \le  &\,   \frac{c_{4,7} }{2}     \int_{\mathbb R}    \left|e^{-\varepsilon  \xi^2}-1  \right|^2 |\xi|^{-2}d\xi\\
   =&\,  \frac{c_{4,7} }{2}         \int_{\mathbb R}    \left|e^{-   \xi^2}-1  \right|^2 |\xi|^{-2}d\xi\cdot \varepsilon^{\frac12}.
    \end{split}
     \end{equation}
     Here, $c_{4,7}$ is a positive constant.

    Putting   \eqref{eq A decomp}, \eqref{eq A 310}, and \eqref{eq A 320} together, we obtain     \eqref{eq J1 1}.      The proof of part   (b) in Lemma  \ref{lem approximation} is complete.
    \end{proof}

    \subsection{Proof of   (c) in Lemma  \ref{lem approximation}}
    \begin{proof}[Proof of   (c) in Lemma  \ref{lem approximation}]
     Since $u\big(t(\e), x\big)$ is $\mathcal F_{t(\e)}$-measurable,
         $\sigma\big(u(t(\e), x)\big)$ can be put  inside the stochastic integral $\mathcal J_{3, \theta}'$, that is
         $$
        \mathcal J_{3, \theta}' =\,  \int_{t(\e)}^t\int_{\mathbb R}p_{[t-s, t+\e-s]}(x-y)  \sigma\left(u(t({\e}), x)\right) W(ds,dy).
    $$
     Consequently,  the   term $\mathcal J_{3, \theta}-\mathcal J_{3, \theta}'$ takes the form
    $$
     \int_{t(\varepsilon)}^{t } \int_{\mathbb R} k_{t, x, \varepsilon}(s,y) W(ds,dy)
    $$
    with
    $$
      k_{t, x, \varepsilon}(s,y) :=  p_{[t-s,\, t+\varepsilon-s]}(x-y) \left[\sigma\left(u(s, y)\big)- \sigma\big(u(t(\varepsilon), x)\right)\right].
    $$
      Since  $k_{t, x, \varepsilon}$ is predictable, applying  the BDG inequality  \eqref{eq BDG},     we obtain
     \begin{equation*} \label{eq A 2}
    \begin{split}
     & \left\|\mathcal J_{3, \theta}-\mathcal J_{3, \theta}' \right\|_{L^{p}(\Omega)}^2\\
    \le & \,  4pc_{2,2}^2 \int_{t(\varepsilon)}^t ds \int_{\mathbb R}dy  \int_{\mathbb R}dz   \left\|  k_{t, x, \varepsilon}(s,y) - k_{t, x, \varepsilon}(s,z)  \right\|^2_{L^{p}(\Omega)}
     |y-z|^{2H-2}.
     \end{split}
    \end{equation*}
     A simple decomposition of the increments  $k_{t, x, \varepsilon}(s,y) - k_{t, x, \varepsilon}(s,z)$   yields 
     \begin{equation}\label{eq A decomp1}
     \left\|\mathcal J_{3, \theta}-\mathcal J_{3, \theta}' \right\|_{L^{p}(\Omega)}^2\le  8pc_{2,2}^2\left[  \int_{t(\varepsilon)}^t  J_{5}(s)ds + \int_{t(\varepsilon)}^t  J_{6}(s)ds\right],
    \end{equation}
     where    
    \begin{equation*} \label{eq A 11}
    \begin{split}
     J_{5}(s):=&\, \int_{\mathbb R}dy  \int_{\mathbb R}dz  \Big| p_{[t-s,\, t+\varepsilon-s]}(x-y)- p_{[t-s,\, t+\varepsilon-s]}( x-z)  \Big|^2\\
    &  \ \,\,\,  \cdot  |y-z|^{2H-2} \left\|\sigma\big(u(s, y)\big)- \sigma\big(u(t(\varepsilon), x)\big)\right\|_{L^{p}(\Omega)} ^2,\\
     J_{6}(s):=&  \, \int_{\mathbb R}dy  \int_{\mathbb R}dz\,  p_{[t-s,\, t+\varepsilon-s]}( x-z)^2
       |y-z|^{2H-2} \left\|\sigma\big(u(s, y)\big)- \sigma\big(u(s, z)\big)\right\|_{L^{p}(\Omega)} ^2.
     \end{split}
    \end{equation*}

    By  \eqref{eq bound} and  \eqref{eq Holder},  there exists  a constant $c_{4,8}\in (0,\infty)$ satisfying  that for any $\vartheta\in (0, \vartheta_0]$,
        \begin{equation*} 
    \begin{split}
      J_{5}(s)\le&\  c_{4,8}   \left(s-t(\varepsilon)\right)^{\vartheta}
     \int_{\mathbb R}dy  \int_{\mathbb R}dz    \left| p_{[t-s,\, t+\varepsilon-s]}(x-y)- p_{[t-s,\, t+\varepsilon-s]}( x-z)  \right|^2 |y-z|^{2H-2} \\
    &\,+ c_{4,8}   \int_{\mathbb R}dy  \int_{\mathbb R}dz   \left| p_{[t-s,\, t+\varepsilon-s]}(x-y)- p_{[t-s,\, t+\varepsilon-s]}( x-z)  \right|^2  |y-z|^{2H-2}  |y-x|^{2\vartheta}\\
    =:&\, c_{4, 8}\left[J_{5,1}(s)+ J_{5,2}(s)\right].
    \end{split}  \end{equation*}

       Since $|s-t(\varepsilon)|\le \varepsilon^{\theta}$ for any $s\in [t(\varepsilon), t]$, we  appeal to Plancherel's identity and a change of variables   to get
     \begin{equation} \label{eq A 112}
    \begin{split}
    & \int_{t(\varepsilon)}^t  J_{5,1}(s)ds\\
     \le&\, \varepsilon^{\theta  \vartheta} \int_{t(\varepsilon)}^t ds   \int_{\mathbb R}dy  \int_{\mathbb R}dz    \Big| p_{[t-s,\, t+\varepsilon-s]}(x-y)- p_{[t-s,\, t+\varepsilon-s]}( x-z)  \Big|^2  |y-z|^{2H-2}\\
    \le  &\,   \varepsilon^{\theta \vartheta} \int_{t(\varepsilon)}^t ds   \int_{\mathbb R}d\xi  \int_{\mathbb R}dh\,  e^{-2(t-s)\xi^2} \left|e^{- \varepsilon \xi^2}- 1 \right|^2  \left| e^{-i \xi h}-1 \right|^2|h|^{2H-2}\\
    \le  &\,   \varepsilon^{\theta \vartheta} \int_{t(\varepsilon)}^t ds   \int_{\mathbb R}d\xi      e^{-2(t-s)\xi^2} \left|e^{- \varepsilon \xi^2}- 1 \right|^2 |\xi|^{1-2H} \cdot  \int_{\mathbb R} \left| e^{-i  h}-1 \right|^2 |h|^{2H-2} dh \\
    =  &\, \frac{1}{2}  \varepsilon^{\theta \vartheta}   \int_{\mathbb R}d\xi   \left(1-e^{-2\varepsilon^{\theta}\xi^2}\right)  \left|1- e^{- \varepsilon  \xi^2} \right|^2 |\xi|^{-1-2H}    \cdot  \int_{\mathbb R} \left| e^{-i  h}-1 \right|^2 |h|^{2H-2} dh \\
    \le &\, \frac{1}{2}  \varepsilon^{\theta \vartheta}   \int_{\mathbb R}d\xi     \left|1- e^{- {\varepsilon } \xi^2} \right|^2 |\xi|^{-1-2H}    \cdot  \int_{\mathbb R} \left| e^{-i  h}-1 \right|^2 |h|^{2H-2} dh \\
      \le \, &   c_{4,9} \varepsilon^{H+\theta \vartheta}.
    \end{split}  \end{equation}
    Here, $c_{4,9}\in (0,\infty)$ is a constant.

     Using a change of variables, we obtain
    \begin{align*}
    J_{5,2}(s) =&\, \int_{\mathbb{R}}dy\int_{\mathbb{R}}dz\left|p_{[t-s,t+\varepsilon-s]}(x-y)-p_{[t-s,t+\varepsilon-s]}(x-z)\right|^2|y-z|^{2H-2}|y-x|^{2\vartheta}\\
=&\,\int_{\mathbb{R}}dy\int_{\mathbb{R}}dh\left|p_{[t-s,t+\varepsilon-s]}(y)-p_{[t-s,t+\varepsilon-s]}(y+h)\right|^2|h|^{2H-2}|y|^{2\vartheta}\\
    \le &\, \int_{\mathbb{R}}dy\int_{\mathbb{R}}dh\left|p_{t-s }(y)-p_{t-s}(y+h)\right|^2|h|^{2H-2}|y|^{2\vartheta}\\
&+\int_{\mathbb{R}}dy\int_{\mathbb{R}}dh\left|p_{t+\varepsilon-s}(y)-p_{t+\varepsilon-s}(y+h)\right|^2|h|^{2H-2}|y|^{2\vartheta}.
    \end{align*}
       By \eqref{eq scaling} and a change of variables, we have that for any  $0\le \vartheta< \frac{1}{2}-H$ and $z\in \{0, \varepsilon\}$,
\begin{align*}
&\int_{\mathbb{R}}dy\int_{\mathbb{R}}dh\left|p_{t+z-s}(y+h)-p_{t+z-s}(y)\right|^2|h|^{2H-2}|y|^{2\vartheta}\\
=&\, (t+z-s)^{-1}\int_{\mathbb{R}}dy\int_{\mathbb{R}}dh\left|p_{1}\left((t+z-s)^{-\frac1{2}}(y+h)\right)-p_1\left((t+z-s)^{-\frac1{2}}y\right)\right|^2
|h|^{2H-2}|y|^{2\vartheta}\\
=&\, (t+z-s)^{\vartheta+H-1  }\int_{\mathbb{R}}dy\int_{\mathbb{R}}dh\left|p_1(y+h)-p_1(y)\right|^2|h|^{2H-2}|y|^{2\vartheta}\\
 \le &\, c_{4,10} (t+z-s)^{\vartheta+H-1},
\end{align*}
where  $c_{4,10} $  is a finite constant by  Lemma \ref{lem green 2}.  Hence, for $0\le \vartheta< \frac{1}{2}-H$ and $\vartheta\le \vartheta_0$,
we have
 \begin{align} \label{eq A 111-2}
\int_{t(\varepsilon)}^t  J_{5,2}(s)ds  \le c_{4,11}\left( \left(\varepsilon+\varepsilon^{\theta}\right)^{\vartheta+H}+\varepsilon^{\theta\left(\vartheta+H\right)}\right)
 \le 3c_{4,10} \varepsilon^{\theta\left(\vartheta+ H\right)}.
\end{align}

  Using the same argument as that in the proof of  \eqref{eq A 320},     we have
    \begin{equation} \label{eq A 22-2}
    \begin{split}
     \int_{t(\varepsilon)}^t J_{6}(s) ds\le    \,  c_{4,11}  \varepsilon^{\frac12},
     \end{split}
    \end{equation}
    where $c_{4,11}\in (0,\infty)$.

Putting   \eqref{eq A decomp1}, \eqref{eq A 112} \eqref{eq A 111-2} and   \eqref{eq A 22-2} together, we obtain \eqref{eq J 3 est}       for any $\vartheta\in (0,  \vartheta_0]\cap \left(0,  \frac12-H\right)$.    The  proof of part (c) in Lemma  \ref{lem approximation} is complete.
                    \end{proof}

    \subsection{Proof of (d) in Lemma  \ref{lem approximation}}
\begin{proof}[Proof of (d) in Lemma  \ref{lem approximation}]
     By the Cauchy-Schwarz inequality, \eqref{eq Holder} and the Lipschitz continuity of $\sigma$, we have
  \begin{equation}\label{eq J1a 6}
     \begin{split}
          &  \left\|\mathcal J_{3, \theta}'-   \widetilde{\mathcal J}_{3,\theta} \right\|_{L^{p}(\Omega)}^2\\
      \le &\,  \left\|\sigma(u(t, x))-\sigma\left(u\left(t(\varepsilon), x\right)\right)  \right\|_{L^{2p}(\Omega)}^2\\
      &\,\,\cdot      \left\|\int_{t(\varepsilon)}^t\int_{\mathbb R}\left[p_{t+\varepsilon-s}(x-y)-p_{t-s}(x-y)\right] W(ds,dy)\right\|_{L^{2p}(\Omega)}^2\\
      \le &\, c_{4, 12}  L_{\sigma}^2 \varepsilon^{2\vartheta_0\theta}  \left\|\int_{t(\varepsilon)}^t\int_{\mathbb R}\left[p_{t+\varepsilon-s}(x-y)
      -p_{t-s}(x-y)\right] W(ds,dy)\right\|_{L^{2p}(\Omega)}^2,
            \end{split}
     \end{equation}
     where  $c_{4,12}>0$.

     By \eqref{eq BDG} and Lemma \ref{lem int p}, we have
    \begin{equation}\label{eq J1a 7}
     \begin{split}
     & \left\|\int_{t(\varepsilon)}^t\int_{\mathbb R}\left[p_{t+\varepsilon-s}(x-y)
      -p_{t-s}(x-y)\right] W(ds,dy)\right\|_{L^{2p}(\Omega)}^2\\
      \le &\, c_{4,13}\int_{t(\varepsilon)}^t ds   \int_{\mathbb R}dy  \int_{\mathbb R}dz   \big|p_{t+\varepsilon-s}(  x-y)- p_{t+\varepsilon-s}(  x-z) +p_{t-s}(  x-z)-p_{t-s}(  x-y) \Big|^2 |y-z|^{2H-2} \\
      \le &\,2c_{4,13} \int_{t(\varepsilon)}^t ds   \int_{\mathbb R}dy  \int_{\mathbb R}dz   \big|p_{t+\varepsilon-s}(  x-y)- p_{t+\varepsilon-s}(  x-z)\Big|^2 |y-z|^{2H-2} \\
      &+\,2c_{4,13}\int_{t(\varepsilon)}^t ds   \int_{\mathbb R}dy  \int_{\mathbb R}dz   \big|p_{t-s}(  x-y)- p_{t-s}(  x-z)\Big|^2 |y-z|^{2H-2}\\
      \le &\, c_{4,14}\int_{t(\varepsilon)}^t (t+\varepsilon-s)^{H-1}ds+ c_{4,14 }\int_{t(\varepsilon)}^t (t-s)^{H-1}ds\\
     \le &\, \frac{2c_{4,14} }{H}  \varepsilon^{H\theta}.
      \end{split}
     \end{equation}
     Here, $c_{4,13}, c_{4,14} >0$.
     
  Putting \eqref{eq J1a 6} and \eqref{eq J1a 7} together, we get \eqref{eq J14 est}.
              The proof of part (d) in Lemma \ref{lem approximation} is complete.
         \end{proof}

 \section{Proof  of Proposition \ref{coro t LIL} }
 
   \subsection{Proof  of  Proposition \ref{coro t LIL}}
 Using the arguments in \cite{WX2024} and Theorem \ref{thm main}, we can obtain  Proposition \ref{coro t LIL}.
More precisely,     the results of   Proposition \ref{coro t LIL} at $t>0$  follow  from Lemma  \ref{lem interpolation}, the decomposition \eqref{eq decom},     and  Khinchin's and Chung's laws of the iterated logarithm for fractional Brownian motions (e.g., \cite[Section 7]{LS01}, \cite{MR1995}, \cite{Tal96}). The results  at  the origin  follow  from  Lemma  \ref{lem interpolation},  Proposition \ref{coro t LIL0}, and the remarks in Section \ref{sec rem initial}.

   \begin{lemma}\label{lem interpolation}  For every $t\in [0,T]$, $x\in \mathbb R$, and $\eta \in \mathcal I$,   with probability one,  
   \begin{equation}\label{eq  interpolation}
   \begin{split}
  & \sup_{0< \varepsilon<\delta} \left|\big(\mathcal D_{\varepsilon} u\big)(t,x) -\left[ p_{t+\e}(\cdot)-p_t(\cdot)\right]*u_0(x)-\sigma(u(t,  x))\big(\mathcal D_{\varepsilon} v\big)(t,x)\right|\\
   =&\, o\left(\delta^{\frac{H}{2}+\eta}\right), \ \ \ \ \text{  as }  \delta\downarrow 0.
   \end{split}
   \end{equation} 
       \end{lemma}
   \begin{proof}
   Similarly to the proof of     \cite[Proposition 3.2]{HK2017} or \cite[Lemma 4.4]{WX2024},  the statement in  \eqref{eq  interpolation} can be proved by Theorem \ref{thm main} and an interpolation argument. For completeness, we provide a proof. 
     
     For any fixed $t\in [0,T]$, $x\in \mathbb R$,   $\varepsilon, \delta\in (0,1)$, and $\eta \in \mathcal I$,  let
     $$
     \Delta(\varepsilon):=\big(\mathcal D_{\varepsilon} u\big)(t, x) -\left[ p_{t+\e}(\cdot)-p_t(\cdot)\right]*u_0(x)-\sigma(u(t, x))\big(\mathcal D_{\varepsilon} v\big)(x), 
          $$
     and
     $$
     \mathscr X(\delta):=\left\{j\delta^{1+ 2\eta/H}:\, 0\le j<\delta^{-\frac{2\eta}{H}}, \, j\in \mathbb Z\right\}.
     $$
    For any $\varepsilon\in \mathscr X(\delta)$, by Theorem \ref{thm main}, we know that for any $b\in (0,H/2+\eta)$,
     $$
     \mathbb P\left(\max_{\varepsilon\in \mathscr X(\delta)}\big|\Delta(\varepsilon)\big| >\delta^{b}  \right)\le \delta^{-kb}
     \sum_{\varepsilon\in \mathscr X(\delta)}\left\|\Delta(\varepsilon)\right\|_k^k\le c_{5,1} \delta^{  \left(\frac{H}{2}+\eta-b\right)k-\frac{2\eta}{H}},    $$
       where the finite positive constant $c_{5,1}$ does not depend on $\delta\in (0,1)$. Therefore,
   \begin{align}\label{eq P123}
     \mathbb P\left(\sup_{\varepsilon\in (0,\delta)}\big|\Delta(\varepsilon)\big| >3\delta^{b}  \right)
     \le   c_{5,1} \delta^{  \left(\frac{H}{2}+\eta-b\right)k-\frac{2\eta}{H}}+\mathcal P_1(\delta)+\mathcal P_2(\delta),
     \end{align}
     where
     \begin{align*}
     \mathcal P_1( \delta) :=&\,  \mathbb P\left(\max_{\varepsilon\in (0,\delta)\atop \varepsilon'\in \mathscr X(\delta), |\varepsilon-\varepsilon'|
     \le \delta^{ 1+2\eta/H}}\big|u(t+\varepsilon, x) -u(t+\varepsilon', x) \big| > \delta^{b}  \right);\\
     \mathcal P_2( \delta):=&\,  \mathbb P\left(\max_{\varepsilon\in (0,\delta)\atop \varepsilon'\in \mathscr X(\delta), |\varepsilon-\varepsilon'|
     \le \delta^{1+2\eta/H}}|\sigma(u(t, x))|\cdot \big|v(t+\varepsilon, x) -v(t+\varepsilon', x) \big| > \delta^{b}  \right).
     \end{align*}
     By  \eqref{eq Holder} and  Kolmogorov's continuity theorem (see, \cite[Theorem C.6,  pp. 107]{K2014}), for any
     $\rho\in (0, \frac{H}{2}+ \eta-b)$ and $k> (1+\frac{2\eta}{H})/\rho$,  there exists  a constant $c_{5,2}>0$ such that
     $$
     \left\| \max_{\varepsilon\in (0,\delta)\atop \varepsilon'\in \mathscr X(\delta), |\varepsilon-\varepsilon'|\le \delta^{1+ 2\eta/H}}\big|\Phi(t+\varepsilon, x) 
     -\Phi(t+\varepsilon', x) \big|\right\|_{L^k(\Omega)}\le c_{5,2}\delta^{H/2+\eta-\rho},
     $$
   for both possible choices of  $\Phi\in \{u, v\}$.     Additionally,  by    H\"older's inequality, \eqref{eq bound} and  the  Lipschitz property of
   $\sigma$, we also have
          $$
     \left\| \max_{\varepsilon\in (0,\delta)\atop \varepsilon'\in \mathscr X(\delta), |\varepsilon-\varepsilon'|\le \delta^{1+2\eta/H}}|\sigma(u(t, x))|
     \cdot \big|v(t+\varepsilon, x) -v(t+\varepsilon',x) \big|\right\|_{L^k(\Omega)}\le c_{5,3} \delta^{H/2+\eta-\rho},
     $$
     where $c_{5,3}\in (0,\infty)$.  As a consequence, it can be deduced from Chebyshev's inequality that for $i=1,2$,
\begin{align}\label{eq P23}
      \mathcal P_i( \delta)\le c_{5,4}\delta^{k(H/2+\eta-b-\rho)},
 \end{align}
 where $c_{5,4}\in (0,\infty)$.   
 
 By \eqref{eq P123} and \eqref{eq P23}, there exists  a finite positive constant $c_{5,5}$, which does not depend on  $\delta\in(0,1)$,  such that
 \begin{align}\label{eq P24}
         \mathbb P\left(\sup_{\varepsilon\in (0,\delta)}\big|\Delta(\varepsilon)\big| >3\delta^{b}  \right)\le c_{5, 5}\left( \delta^{  \left(\frac{H}{2}+\eta-b\right)k-\frac{2\eta}{H}} +\delta^{k(H/2+\eta-b-\rho)}\right).
 \end{align}
  Choose $k> \max\{ \frac{2\eta}{H}\left(\frac{H}{2}+\eta-b\right),  (1+\frac{2\eta}{H})/\rho \}$.  Replacing $\delta$ by $2^{-n}$ and applying a monotonicity argument with the Borel-Cantelli lemma, we conclude that, almost surely,
 $$
 \sup_{\varepsilon\in (0,\delta)}\big|\Delta(\varepsilon)\big| =O( \delta^{b})\ \ \ \text{as } \delta\downarrow0.
 $$
    Finally, \eqref{eq  interpolation} follows from the arbitrariness of $b\in (0,H/2+\eta)$ and $\eta\in \mathcal I$. The proof is complete.      
     \end{proof}

    By Theorem \ref{thm main}  and using similar arguments from \cite{FKM2015, WX2024}, we obtain    Khinchin's law of the iterated logarithm,  Chung's law of the iterated logarithm   for $u(t,x)$ at any fixed $t>0$ and  $x\in \mathbb R$.
    For the  LILs  at $t=0$, we need to establish the associated  results for the linear SHE by using an approach in Lee and Xiao \cite{LX2023}.
 \subsection{Harmonizable representation for the solution}
Recall the linear SHE $v=\{v(t,x)\}_{t\ge0, x\in \mathbb R}$ given in \eqref{eq SHE linear}.     
  By     \eqref{CovStru}    and the change of variables
$\tau:= 2(t-s) |\xi|^{2}$, we have
\begin{equation}\label{eq v moment}
\begin{split}
 \EE\left[|v(t,x)|^2\right]
=&\,c_{1,H}\int_0^t\int_{\RR}e^{-2(t-s)|\xi|^{2}}|\xi|^{1-2H}d\xi ds \\
=&\, \frac{c_{1,H}}{2^{1-H}}  \int_0^\infty e^{-\tau }\tau^{-H} d\tau \int_0^{t} (t-s)^{H-1} ds \\
=&\, \frac{c_{1,H} }{2^{1-H} H}  \Gamma\left(1-H\right) t^{H}\\
=&\,   \widetilde\kappa^2t^{H},
\end{split}
\end{equation}
where $  \widetilde\kappa$ is given by  \eqref{eq con kappa1}. Furthermore, for any $t, s\ge0, x, y\in \mathbb R$, we have
\begin{equation}\label{eq v  holder}
 \left\|v(t,x)-v(s,y)\right\|_{L^2}\le  c_{5,6}\left( |t-s|^{H/2}+|x-y|^H\right),
\end{equation}
where $c_{5,6}\in(0,\infty)$. See \cite[Theorem 1.1]{BJQ2016} or \cite[Lemma 2.1]{LQW25}.

Let $G_{t, x}(s, y)=p(t-s, x-y)\mathbf 1_{[0,t]}(s)$. It can be verified that 
 the Fourier transform of $G_{t, x}(\cdot, \cdot)$ is
\begin{equation}\label{Eq:FT_g}
\mathcal{F}G_{t, x}(\tau, \xi) =  \frac{e^{-i \xi x}\left(e^{-i\tau t} - e^{-t |\xi|^2}\right)}
{|\xi|^2 - i\tau}, \quad \tau \in \R,\ \xi \in \R .
\end{equation}
  Let $\widetilde W_1$ and $\widetilde W_2$ be two independent space-time Gaussian white 
noises on $\R\times \R $, and  let $\widetilde W = \widetilde W_1 + i \widetilde W_2$.
For each $(t, x) \in \mathbb R_+\times \R $,  define 
   the Gaussian random field $\{v_x(A, t), A \in \mathscr{B}(\R_+)\}$ 
  by
\begin{equation}\label{Eq:SHE-HR}
v_x(A, t) := c_{1,1}^{1/2} \,   \iint_{\{(\tau, \xi) : \max(|\tau|^{H/2}, |\xi|^{H}) \in A\}}
\mathcal{F}G_{t, x}(\tau, \xi)\,  |\xi|^{\frac12-H}
\widetilde W(d\tau, d\xi).
\end{equation}

\begin{lemma}\label{Lem:SHE-a1} 
The  Gaussian random field $\{v_x(A, t), A \in \mathscr{B}(\R_+), t >0, x\in \mathbb R \}$ given by  \eqref{Eq:SHE-HR}  satisfies the following conditions:
\begin{itemize}
\item[(a)] For every $t>0$ and $x\in \mathbb R$, $A \mapsto v_x(A, t)$ is an independently scattered Gaussian 
noise such that  the processes $v_x(A,  t)$ and $v_x(B, t)$ are independent whenever $A$ and $B$ 
are disjoint. $\{v_x(\mathbb{R}_+, t)\}_{t\ge0, x\in \mathbb R}$ has the same law with   $\{v(t,x)\}_{t\ge0, x\in \mathbb R}$ defined by \eqref{eq SHE linear}.
 \item[(b)] 
There exists a finite constant $c_{5,7}>0$ such that for all $0 \le a < b \le \infty$,
  $t>0, x\in \mathbb R$,
\begin{align}
\begin{aligned}\label{Eq:SHE-a1}
 {\| v_x(\mathbb R_+, t)-v_x([a, b), t)   \|}_{L^2} 
 \le c_{5,7} \left( a^{\frac{2}{H}-1}  t   + b^{-1}\right).
\end{aligned}
\end{align} 
\end{itemize}
\end{lemma} 
\begin{proof} 
For Part  (a),  it is obvious that for every $t>0$ and $x\in \mathbb R$, $A \mapsto v_x(A, t)$   is an independently scattered Gaussian 
noise.   By \eqref{CovStru}, for any $(t, x), (s, y) \in \mathbb R_+ \times \R $,  
\begin{align*}
\E[v_x(\mathbb{R}_+, t) v_y(\mathbb{R}_+, s)] 
 = c_{1,1} \iint_{\R \times \R } \mathcal{F}G_{t, x}(\tau, \xi)\, 
\overline{\mathcal{F}G_{s, y}(\tau, \xi)}\,  |\xi|^{1-2H} \, d\tau \, d\xi 
  = \E[v(t, x) v(s, y)],
\end{align*}
where  $\{v(t,x)\}_{t\ge0, x\in \mathbb R}$ is defined by \eqref{eq SHE linear}. 
Consequently, the processes   $\{v_x(\mathbb{R}_+, t)\}_{t\ge0, x\in \mathbb R}$ and  $\{v(t,x)\}_{t\ge0, x\in \mathbb R}$ are identical in law.

 The proof of  Part (b)   is similar to that of Lemma 7.3 in \cite{DMX17}.
First,
\begin{align*}
 v_x(\mathbb R_+, t)   - v_x([a, b), t) 
 =    v_x([0, a), t)  + v_x([b, \infty), t).
\end{align*}
By   \eqref{Eq:FT_g}, we have
\begin{align}
\begin{split}\label{Eq:a}
 \E\left[ v_x([0, a), t)^2\right] 
 = c_{1,1} \iint_{D_1(a)} \frac{\varphi_1(t, \tau, \xi)^2 + 
\varphi_2(t, \tau)^2}{|\xi|^4 + |\tau|^2}  |\xi|^{1-2H}\, d\tau\, d\xi,
\end{split}
\end{align}
\begin{align}
\begin{split}\label{Eq:b}
 \E\left[ v_x([b, \infty), t) ^2\right] 
  = c_{1,1} \iint_{D_2(b)} \frac{\varphi_1(t, \tau, \xi)^2 + \varphi_2(t, \tau)^2}
{|\xi|^4 + |\tau|^2}  |\xi|^{1-2H}\, d\tau\, d\xi,
\end{split}
\end{align}
where
\begin{align*}
& D_1(a) 
:= \left\{ (\tau, \xi) \in \R \times \R  : \max\left(|\tau|^{H/2}, |\xi|^{H}\right) < a \right\},\\
& D_2(b) 
:=\left \{ (\tau, \xi) \in \R \times \R  : \max\left(|\tau|^{H/2}, |\xi|^{H}\right) \ge b \right\},\\
&\varphi_1(t, \tau, \xi) 
:= \cos(\tau t) - e^{-t|\xi|^2},\\
&\varphi_2(t, \tau)
:= -\sin(\tau t).
\end{align*} 
For \eqref{Eq:a},  noticing  that  
\begin{align*}
|\partial_t \varphi_j| \le |\tau| + |\xi|^2, \quad j = 1, 2.
\end{align*}
  by the mean value theorem, we have
\begin{equation} \label{BD:a} 
\begin{split}
 \E\left[(v_x([0, a), t))^2\right] 
& \le  4 c_{1,1}  t^2 \iint_{D_1(a)}   |\xi|^{1-2H} d\tau\, d\xi\\
 & =  8 c_{1,1}t^2    
  \int_{-a^{2/H}}^{a^{2/H}} d\tau \,  
\int_0^{a^{1/H}} dr \, r^{1-2H}\\
& =\frac{8 c_{1,1}}{1-H}   t^2  a^{ \frac{2(2-H)}{H}}.
\end{split}
\end{equation} 
 
For \eqref{Eq:b}, we can use the bounds $|\varphi_1| \le 2$ and $|\varphi_2| \le 1$ 
to deduce that
\begin{equation} \label{BD:b}
\begin{split}
&\E\left[v_x([b, \infty), t) ^2\right]\\
\le &\, 5 c_{1,1} \iint_{D_2(b)} \frac{        |\xi|^{1-2H}}{ |\xi|^{4} + |\tau|^2}
d\tau\, d\xi\\
  \le &\, 10c_{1,1} \iint_{\{ |\tau| \le r^{2}, r  \ge b^{1/H} \}} 
\frac{      r^{1-2H}}{ r^{4}} d\tau\, dr
+ 10c_{1,1} \iint_{\{ |\tau|^{1/2}  > r>0, |\tau| \ge b^{2/H}  \}}
\frac{|r|^{1-2H}}{|\tau|^2} d\tau\, dr\\
 = &\,  20c_{1,1}\int_{b^{1/H}}^\infty dr\, r^{-3-2H} 
\int_{0}^{r^2} d\tau\,       
+ 20c_{1,1} \int_{b^{2/H}}^\infty d\tau \,  \tau^{-2} 
\int_0^{\tau^{1/2}} dr\, r^{1-2H}\\
 = &\, \frac{10 (2-H)}{H(1-H)} c_{1,1} b^{-2}.
 \end{split}
\end{equation}
  Therefore, \eqref{Eq:SHE-a1} follows immediately from \eqref{BD:a} and \eqref{BD:b}.
     The proof is complete. 
\end{proof}

\subsection{LILs of the linear SHE at zero}\label{LILs at origin}
      \begin{proposition}\label{coro t LIL0}  
        Choose and fix   $x\in \mathbb R$.   Then,  
        \begin{itemize}
        \item[(a)] (Khinchin's LIL)
        \begin{equation}\label{Eq:LIL0}
        \begin{split}
        \lim_{r \to 0}  \sup_{ 0<t<r}  \frac{|v(t, x)|}{t^{H/2}\sqrt{2 \log\log(1/t)}}=  \widetilde\kappa, \ \ \ \text{a.s.}, 
        \end{split}
                 \end{equation}
where $  \widetilde\kappa$ is given by  \eqref{eq con kappa1}.
            \item[(b)]  (Chung's LIL)
        \begin{equation}\label{Eq: CLIL0}
        \begin{split}
        \liminf_{\e \to 0}\sup_{0\le r\le \e}\frac{|v(r, x)|}{(\e/  \log\log(1/\e))^{H/2}}=  \left(  \frac{ \Gamma(2H)}{ \Gamma(H)} \right)^{\frac12}  \lambda_H^{H/2},    \ \ \ \text{a.s.}
        \end{split}
        \end{equation}
        where $\lambda_H$ is the small ball constant of a  fBm with index $\frac{H}{2}$ (see, e.g., \cite[Theorem 6.9]{LS01}).
      \end{itemize}
 \end{proposition}
\begin{proof}  
Chung's LIL at the origin in Part (b) was established in \cite{LW25} via an adaptation of the method introduced in \cite{KKM23}.  Next, we are going to prove   Khinchin's LIL at the origin in Part (a).
Fix $x \in \mathbb R$.  For any    $r > 0$, define 
$$L(r) := \sup_{ 0<t<r} \frac{|v(t, x)  |}
{t^{H/2}\sqrt{2 \log\log(1/t)}}. $$ 
To prove \eqref{Eq:LIL0}, we claim that
\begin{equation}\label{LIL:UB}
\lim_{r \to 0+} L(r) \le    \widetilde\kappa, \quad \text{a.s.},
\end{equation}
\begin{equation}\label{LIL:LB}
\lim_{r \to 0+} L(r) \ge   \widetilde\kappa, \quad \text{a.s.}
\end{equation}
We  prove these two bounds separately. 

\noindent({\bf Upper bound}). 
Let $a > 1$ and $\zeta > 0$ be constants. 
For each $n \ge 1$, let 
$$r_n:= a^{-n} \quad \text{and} \quad \theta_n := (1+\zeta)  \widetilde\kappa r_n^{H/2} \sqrt{2 \log\log(1/r_n)}.$$
Consider the event
$$A_n := \bigg\{ \sup_{ 0\le t \le r_n } |v(t, x)  | > \theta_n \bigg\}. $$
We are going to use Lemma \ref{Tal94} to derive an upper bound for $\mathbb P(A_n)$.
For any  $n\ge1$,  by \eqref{eq v moment}, 
$$\sigma^2 := \sup_{0\le t\le     r_n} \|v(t, x)\|^2_{L^2} =   \widetilde\kappa^2r_n^H,$$ 
and by  \eqref{eq v  holder}, for all $t, s\ge 0$,
$$d_v(t, s) \le c_{5,6}  |t-s|^{H/2}. $$
Then, for any  $0 < \eps < \sigma$,
$$N([0,r_n], d_v, \eps) \le c_{5,8} r_n \eps^{-\frac2H},$$  
where $c_{5,8} $ is a constant independent of $\eps$  and $n$. 

For $n$ large enough, $\theta_n >  \widetilde\kappa  r_n^{H/2}\left(1+\sqrt{2/H}\right)$.
Applying Lemma \ref{Tal94} with $\eps_0 = \sigma$ and $p =\frac{2}{H}$, 
 we have
\begin{align*}
\mathbb P(A_n) &\le 2 \left( \frac{Kc_{5,3}^{H/2}  \theta_n}{\sqrt{2/H}\,   \widetilde\kappa r_n^{H/2}}\right)^{2/H} \Phi\left(\frac{\theta_n}{  \widetilde\kappa r_n^{H/2} }\right),
\end{align*}
where $\Phi(x) = (2\pi)^{-1/2} \int_x^\infty e^{-y^2/2} \, dy$.
Using the estimate \eqref{stdGaussian},
we get that
\begin{align*}
\mathbb P(A_n) & \le c_{5,9} {(\log n)}^{1/H} \, n^{-(1+\zeta)^2},
\end{align*}
where $c_{5,9}$ is a finite constant. Hence, $\sum_{n=1}^\infty \mathbb P(A_n) < \infty$.
By the Borel--Cantelli lemma,
\begin{align*}
\limsup_{n \to \infty} \sup_{ 0<t<r_n}
\frac{|v(t, x)  |}{r_n^{H/2} \sqrt{2\log\log(1/r_n)}}
\le (1+\zeta) \widetilde\kappa , \quad \text{a.s.}
\end{align*}
 and thus
 $$\limsup_{n \to \infty} \sup_{ 0<t<r_n}
 \frac{|v(t, x)  |}{r_{n+1}^{H/2} \sqrt{2 \log\log(1/r_n)}} \le a^{H/2} (1+\zeta)    \widetilde\kappa \quad \text{a.s.} 
 $$
This implies that
$$\limsup_{r \to 0+} L(r) \le a^{H/2} (1+\zeta)   \widetilde\kappa , \quad \text{a.s.} $$
  Letting $a \downarrow 1$ and $\zeta \downarrow 0$ along rational sequences, 
we get \eqref{LIL:UB}.

\noindent({\bf Lower bound}).
Fix $0 < \eps < 1$.
Let $0 < \delta < 1$ be a small fixed number (depending on $\eps$) to be determined later. For each $n \ge 1$, let 
$$t_n: =   \rho_n^{2/H}, $$
where $$\rho_n := \exp\left(-\left(n^\delta + n^{1+\delta}\right)\right).$$
According to Lemma \ref{Lem:SHE-a1}, we can write $v(t, x) = v_n(t) + \widetilde v_n(t)$, where
$$v_n(t): = v([b_n, b_{n+1}), t), \quad
\widetilde{v}_n(t) := v_x(\R_+ \setminus [b_n, b_{n+1}), t),$$ 
and $b_n := \exp\left(n^{1+\delta}\right)$.

We aim to prove that
\begin{equation}\label{LIL:Eq1}
\limsup_{n \to \infty} \frac{|v_n(t_n)  |}
{t_n^{H/2}\sqrt{2\log\log(t_n^{-1})}} \ge (1-\eps)    \widetilde\kappa ,
\quad \text{a.s.}
\end{equation}
and
\begin{equation}\label{LIL:Eq2}
\limsup_{n \to \infty} \frac{|\widetilde{v}_n(t_n)  |}
{t_n^{H/2}\sqrt{2 \log\log(t_n^{-1})}} \le \eps,  \quad \text{a.s.}
\end{equation}

To prove \eqref{LIL:Eq1},  for each $n \ge 1$ we define the event
$$B_n := \left\{ |v_n(t_n)  | \ge 
(1- \eps)  \widetilde\kappa   t_n^{H/2} \sqrt{2\log\log(t_n^{-1})} \right\}. $$
By   Lemma \ref{Lem:SHE-a1}, for all $t>0$,
\begin{equation}\label{Chung:Eq1}
\| \tilde{v}_n(t) \|_{L^2} 
\le c_{5,7} \bigg(b_n^{\frac{2}{H} - 1}t + b_{n+1}^{-1} \bigg).
\end{equation} 
 Let $D_n$ be the diameter of $S_n :=\left [0,  \rho_n^{H/2}\right]$ in the metric 
$d_{\tilde{v}_n}$. Then
\begin{equation}\label{Chung:Eq2}
D_n \le 
c_{5,7} \rho_n \left(    (b_n \rho_n)^{\frac{2}{H} - 1}  + (b_{n+1}\rho_n)^{-1} \right).
\end{equation}
Note that $b_n \rho_n = \exp\left(-n^\delta\right)$.
Also, by the mean value theorem, 
$$(n+1)^{1+\delta} - n^{1+\delta} \ge (1+\delta)n^\delta,$$  which implies
$b_{n+1} \rho_n \ge \exp\left(\delta n^\delta\right)$.
When $\delta \le 2/H-1$, 
we have
\begin{equation}\label{Chung:Eq3}
D_n \le c_{5,7} \rho_n \exp\left(-\delta n^\delta\right).
\end{equation}

By the triangle inequality,   \eqref{eq v moment} and \eqref{Chung:Eq3}, we have
\begin{equation}\label{L2-vn}
\begin{split}
\|v_n(t_n)  \|_{L^2} 
& \ge  \|v(t_n)  \|_{L^2} - \|\tilde v_n(t_n)   \|_{L^2}\\
& \ge   \widetilde\kappa ^2\left(1 - c_{5,7}\exp\left(-\delta n^\delta\right)\right)t_n^{H/2}.
\end{split}
\end{equation}
Now,  \eqref{L2-vn}   implies  that for $n$ large,
\begin{align*}
B_n \supset \left\{ |v_n(t_n)  | \ge (1-\eps/2)
\|v_n(t_n)  \|_{L^2} \sqrt{2\log\log (1/\rho_n)} \right\}.
\end{align*} 
Then, by the standard Gaussian estimate \eqref{stdGaussian}, 
we get that  for $n$ large,
\begin{align*}
\mathbb P(B_n) & \ge c_{5, 10} (\log n)^{-1/2} n^{-(1-\eps/2)^2(1+\delta)},
\end{align*}
where $c_{5,10}\in (0,\infty)$. 
Choosing $\delta$ small enough such that $(1-\eps/2)^2(1+\delta)\le 1$,
we have $$\sum_{n=1}^\infty \mathbb P(B_n) = \infty.$$
Hence, by the independence among $\{v_n\}_{n\ge1} $ and 
the second Borel--Cantelli lemma, we get \eqref{LIL:Eq1}.

For \eqref{LIL:Eq2}, we use    \eqref{L2-vn}  to get that
\begin{align*}
& \mathbb P\left\{ |\tilde{v}_n(t_n)  | \ge 
\eps t_n^{H/2} \sqrt{2\log\log(t_n^{-1})} \right\}\\
  \le &\, \mathbb P\left\{ |\tilde{v}_n(t_n)  | \ge 
c_{5,7}^{-1}\tilde \kappa^{-1} \eps \| {v}_n(t_n)  \|_{L^2} \exp\left(\delta n^\delta\right)
\sqrt{2\log\log(1/\rho_n)} \right\}.
\end{align*}
This probability, by standard Gaussian estimate \eqref{stdGaussian}, is bounded above by
$$c_{5,11} \exp\left(-\frac{1}{2}c_{5,7}^{-2}\tilde \kappa^{-2} \varepsilon^2  \exp\left(2\delta n^\delta\right) \log\log(C/\rho_n)\right) 
\le c_{5,12}n^{-2}, $$
for $n$ large, where  $c_{5,12}\in (0,\infty)$. Thus, the Borel--Cantelli lemma implies \eqref{LIL:Eq2}.

Since $v(t) = v_n(t) + \widetilde{v}_n(t)$, combining \eqref{LIL:Eq1} and \eqref{LIL:Eq2} yields
$$\limsup_{n \to \infty} \frac{|v(t_n)  |}
{t_n^{H/2} \sqrt{2\log\log\left(t_n^{-1}\right)}} \ge (1-\eps)   \widetilde\kappa - \eps \quad \text{a.s.} 
$$
Letting $\eps \downarrow 0$ along a rational sequence, we get \eqref{LIL:LB}.  
The proof is complete. 
\end{proof}

\section{Quadratic  variation of the solution to the nonlinear SHE}
 
Let $x\in \mathbb R$ be fixed. For every $N\ge1$, let
\begin{align}\label{eq V u}
V_{N, x}(u):= N^{H-1}\sum_{i=0}^{N-1}\big(u(t_{i+1}, x)-u(t_i, x) \big)^2,
\end{align}
with
\begin{align}\label{eq t i}
t_i=\frac{i}{N} \ \ \text{for } i=0,1, \cdots, N.
\end{align}
We are going to  analyze the asymptotic behavior of the sequence $\{V_{N, x}(u)\}_{ N\ge 1}$, as $N\rightarrow\infty$.
  The idea is to approximate the increment $u(t_{i+1},  x)- u(t_i,  x)$  by 
  $$\sigma\big(u(t_i(\varepsilon), x )\big)\big(v(t_{i+1},   x)- v(t_i,  x)\big),$$ where   $v$ is the solution to the linear SHE \eqref{eq SHE linear} and 
  \begin{align}\label{eq tiN}
  t_i(\varepsilon)=t_i- N^{-\theta},
  \end{align}
   with $\theta$ being given by \eqref{eq const a}.   We first analyze the quadratic variation of the solution to the linear SHE, and then estimate the approximation error.

\subsection{Quadratic  variation of the solution to the   stochastic linear heat equation}
    \subsubsection{A decomposition of the stochastic linear heat equation}
          Recall the stochastic convolution $ \{v(t, x)\}_{(t,x)\in \mathbb R_+\times \mathbb R}$ given by \eqref{eq mild SHE}.
     \begin{lemma}\label{lem decom}    There exists a Gaussian process $\{T(t)\}_{t\ge0}$ independent of $ \{v(t, x)\}_{(t,x)\in \mathbb R_+\times \mathbb R}$ such that the following items hold.
     \begin{itemize}
     \item[(a)] The process 
     \begin{align}\label{eq decom}
       X(t):= \kappa^{-1}\left(v(t,x)+T(t) \right) \ \ \ (t\ge0)
       \end{align}
       is a  fBm with Hurst parameter $\frac{H}{2}$, where $\kappa$ is  given in \eqref{eq constant 1}.
       \item[(b)] $\{T(t)\}_{t\ge0}$ is self-similar with  index $\frac{H}{2}$.
       \item[(c)]  $\{T(t)\}_{t\ge0}$ has a version that is infinitely-differentiable on $(0,\infty)$.
    \item[(d)] For every $t>s>0$,
       \begin{equation}\label{eq T moment}
        \mathbb E\left[|T_t-T_s|^2 \right]\le c_{6,1} \frac{(t-s)^2}{s^{2-H}},
    \end{equation}
    where $c_{6,1}=\frac{2^H \Gamma(1+2H)\Gamma(2-H)\sin(H\pi)}{16\pi}$.
      \end{itemize}
    \end{lemma}
    
    \begin{proof}
  While Parts (a)-(c) follow from \cite[Theorem 1]{WZ2021}, which adapts techniques from \cite[Theorem 3.3]{K2014}, Part (d) requires additional quantitative estimates. Below we outline the key steps of this derivation.
        
    Let $M$ denote  a white noise on $\mathbb R$ independent of $W$, and  let   $\{T(t)\}_{t\ge0}$ be a   Gaussian process defined by
     \begin{align}\label{eq T}
     T(t):=\left( \frac{\Gamma(1+2H)\sin(H\pi)}{4\pi}\right)^{\frac12} \int_{-\infty}^{\infty} \left( 1-e^{-t\xi^2}\right)|\xi|^{-\frac12-H} M(\xi), \ \ \ t\ge0.
     \end{align}
    Through a change of variables, we immediately see that $\{T(t)\}_{t \geq 0}$ is self-similar with index $\frac{H}{2}$. Moreover, since $\{T(t)\}_{t \geq 0}$ and $\{v(t, x)\}_{t \geq 0}$ are independent, direct calculation yields that 
      $$
     \mathbb E\left[\Big|\left[v(t+\e, x)+T(t+\e)\right]-\left[v(t, x)+T(t)\right] \Big|^2\right]= \frac{\Gamma(2H)}{\Gamma(H)} \e^{H}.
     $$
     This implies that
     $$
     X(t):=\left(\frac{\Gamma(2H)}{\Gamma(H)}\right)^{-\frac12} \left(v(t,x)+T(t) \right)  \ \ \ (t\ge0)
     $$
     is a fBm with Hurst parameter $\frac{H}{2}$. Using the same argument as that in the proof of \cite[Lemma 3.6]{K2014}, we know that the random function $\{T(t)\}_{t\ge0}$ has a version that is infinitely differentiable on $(0,\infty)$.
    
         By \eqref{eq T} and   the elementary inequality $\left|1-e^{-x}\right|\le x$ for any $x\ge0$,  we have
    \begin{equation*}
    \begin{split}
     \mathbb E\left(|T(t)-T(s)|^2\right)
    =&\, \frac{\Gamma(1+2H)\sin(H\pi)}{4\pi } \int_{\mathbb R} e^{- 2 s |\xi|^{2}}  \left( 1 - e^{-  (t-s)|\xi|^{2} }\right)^2 |\xi|^{-  1-2H}d\xi \\
    \leq &\,  \frac{\Gamma(1+2H)\sin(H\pi)}{2\pi }  (t-s)^2   \int_0^{\infty} e^{- 2 s \xi^{2}}    \xi^{3-2H }d\xi      \\
    = &\,  \frac{\Gamma(1+2H)\sin(H\pi)}{4\pi }(t-s)^2  \left(\frac{1}{2s}\right)^{2-H}    \int_0^{\infty} e^{- \eta}   \eta^{1-H }d\eta,
     \end{split}
     \end{equation*}
     where the change of variables $\eta=2s \xi^2$ is used in the last step.  Hence, \eqref{eq T moment} holds.
    
     The proof is complete.
       \end{proof}

 \subsubsection{Quadratic variation of the perturbed fBm}
 In this part, we recall some results of the quadratic variation for the perturbed fractional Brownian motion, taken from Olivera and Tudor \cite{OT25}.
 
 Fix $x\in \mathbb R$ and let $v(t,x)$ be given by (\ref{pfBm}).  By Lemma~\ref{lem decom} 
$v(t, x)$  can be represented as a perturbed fractional Brownian motion:
  	\begin{equation}\label{pfBm}
		v(t,x)= \kappa B^{H/2}(t)+ T(t),  \ \ \ \ t\geq 0,
	\end{equation}
	where  $\kappa$ is  given in \eqref{eq constant 1}, $B^{H/2}(t)$ is  a  fBm with Hurst parameter $\frac{H}{2}$,  
	and   the process $T(t)$ is a centered Gaussian perturbation term satisfying the properties stated in Lemma~\ref{lem decom}. 
	
	  Recall that $t_{i}, i=0,1,...,N$ are given by \eqref{eq t i}.   	
	 	\begin{lemma}\cite[Lemma 2]{OT25}\label{ll2}
		Fix $x\in \mathbb R$ and let $v(t,x)$ be given by (\ref{pfBm}).   Then, there exists a constant $c_{6,2}>0$ satisfying the following properties.
		\begin{itemize}
		\item[(i)] For every $i\geq 1$,	
		 		 		\begin{equation}\label{2f-2}
			\mathbb{E}\left[ \left( N ^{H-1} \left( v(t_{i+1}, x) - v(t_{i}, x)\right)^{2} -\frac{\kappa^{2}}{N} \right) ^{2}\right]\leq    c_{6,2} N^{-2} ,
		\end{equation}
	\item[(ii)] For every $i,j\geq 1$ with $i\neq j$, 
		 		\begin{equation}\label{14s-1}
				\begin{split}
			 &	\mathbb{E}\left[ \left|\left( N ^{2H-1}\left( v(t_{i+1}, x) - v(t_{i}, x)\right)^{2} -\frac{\kappa^{2}}{N} \right)\left( N ^{2H-1} \left( v(t_{j+1}, x) - v(t_{j}, x)\right)^{2} -\frac{\kappa^{2}}{N} \right)\right|\right]  \\
			  \leq   & \, c_{6,2} N^{-2}  \left(\rho_{H}^{2}(i-j)+ i ^{H/2-1}+j^{H/2-1}\right),  
			\end{split}
		\end{equation}
		with
		\begin{equation} \label{ro}
			\rho_{H}(k):= \frac{1}{2}\left(\vert k+1\vert ^{H}+\vert k-1\vert ^{H}-2\vert k\vert ^{H}\right), \hskip0.4cm k\in \mathbb{Z}.
		\end{equation}
		 \end{itemize}
 \end{lemma}

 According to   \cite[Proposition 2]{OT25}	and Lemma \ref{lem decom}, we have the following result. 
	\begin{proposition}\label{pp22} 	 Fix $x\in \mathbb R$. For $N \geq 1$, let
		  \begin{align}\label{eq V v}
V_{N, x}(v):= N^{H-1}\sum_{i=0}^{N-1}\left(v(t_{i+1}, x)-v(t_i, x) \right)^2.
\end{align}
where  $t_{i}$ is given in \eqref{eq t i}.
		Then, the sequence $\left\{V_{N, x}(v) \right\}_{N\geq 1}$ converges to $\kappa^{2}$ in $L^{p} (\Omega)$ for every $p\geq 1$, with $\kappa $ given in \eqref{eq constant 1}. Moreover, for $N$ large enough, 
		\begin{equation}\label{2f-10}
			\mathbb{E}\left[ \left| V_{N,x}(v) - \kappa ^{2} \right| ^{2} \right] \leq  c_{6,3} N^{-1}, 
					\end{equation}
		where $c_{6,3}\in (0,\infty)$. 
	\end{proposition}

  \subsection{Quadratic  variation of the solution to the nonlinear SHE} 

Now, we prove an approximation result for the temporal increment of the random field   $u(t,x)$.
  \begin{proposition}\label{prop variation converge} Assume    Condition \ref{cond A} holds with $\beta_0>H$. 
Let $V_{N, x}(u)$ be given by \eqref{eq V u}. Then, for every  fixed $x>0$ 
  and   $N\ge 2$,
  \begin{align}\label{eq quad}
  \mathbb E\left[\left| V_{N, x}(u)-  \kappa^2 \int_0^1 \sigma^2\big(u(t, x) \big) dt \right| \right]\le c_{6,4} N^{-\delta},
  \end{align}
  where   the constant  $\kappa$ is given by  \eqref{eq constant 1}, $c_{6,4}>0$ does not depend on $x$ and $N$, and  
 $\delta \in \mathcal I$ which is defined by  \eqref{eq interval} with $\vartheta_0=H/2$.
   \end{proposition}
 \begin{proof}
 The proof follows the approach of \cite[Theorem 1]{OT25}; for completeness, we outline the key steps.
  Given $\varepsilon\in (0,t)$,   we define the modified temporal increment of 
 $v$ as
\begin{equation}\label{eq Dv tilde}
\begin{split}
(\widetilde{\mathcal D}_{\varepsilon} v)(t, x) 
 := &\,   \int_0^{t(\varepsilon)}\int_{\mathbb R} \big(p(t+\varepsilon-s, x-y)- p(t-s, x-y)  \big) \widetilde W(ds,dy)\\
 &   \, + \int_{t(\varepsilon)}^{t } \int_{\mathbb R} \big(p(t+\varepsilon-s, x-y)- p(t-s, x-y)  \big) W(ds,dy)\\
 &\,  + \int_t^{t+\varepsilon} \int_{\mathbb R} p(t+\delta-s, x-y) W(ds,dy),
 \end{split}
\end{equation}
where $\widetilde W$ is an independent copy of the  noise $W$. The modified increment operator $(\widetilde{\mathcal D}_{\varepsilon} v)(t, x)$ equals in distribution to the forward difference 
$({\mathcal D}_{\varepsilon} v)(t, x)  := v(t+\varepsilon,x) - v(t,x)$ and it is independent of the $\sigma$-algebra $\mathcal{F}_{t(\varepsilon)} = \sigma\left\{W([0,s]\times A) : s \leq t(\varepsilon), A \in \mathcal{B}(\mathbb{R})\right\}$.
 In the following, we take    $$\varepsilon=N^{-1}  \ \ \ \text{and} \ \ \   t(\varepsilon)= t-\varepsilon^{\theta},$$ where $\theta$ is given by \eqref{eq const a}.   
        
For every  fixed $x \in \mathbb{R}$,  we have
\begin{align*}
& \,  V_{N, x}(u) -   \kappa^2 \int_0^1 \sigma^2\left(u(t, x) \right) dt \\
=  
  & \,    N^{H-1}\sum_{i=0}^{N-1} \left[ \left(u(t_{i+1}, x)-u(t_i, x) \right)^2 -  \sigma^2\left(u\left(t_i(\varepsilon), x\right)\right) \left(\widetilde{\mathcal D}_{\varepsilon}v\right)^2\left(t_i, x\right)  \right] \\
  & \, + \sum_{i=0}^{N-1} \sigma^2\left(u(t_i(\varepsilon), x)\right) \left[ N^{H-1}  \left(\widetilde{\mathcal D}_{\varepsilon}v\right)^2\left(t_i, x\right) -  \kappa^2  N^{-1}\right] \\
  & \, +   \kappa^2 \left[N^{-1} \sum_{i=0}^{N-1} \sigma^2\left(u\left(t_i(\varepsilon), x\right)\right) -  \int_0^1 \sigma^2\left(u\left(t, x\right) \right) dt \right]\\
    =:& \,    L_{1, N} + L_{2, N} + L_{3, N}.
\end{align*} 
We analyze the three summands separately. 

Applying the Cauchy-Schwarz inequality and \eqref{eq main 21}, we  have
\begin{equation}\label{eq B1N-1}
\begin{split}
   \mathbb{E}\left[|L_{1, N}|\right] 
\leq & \, N^{H-1} \sum_{i=0}^{N-1}   \left\| \left(u(t_{i+1}, x)-u(t_i, x) \right) -  \sigma\left(u(t_i(\varepsilon), x)\right) \left(\widetilde{\mathcal D}_{\varepsilon}v\right)\left(t_i, x\right) \right\|_{L^2}  \\
&\, \times \,     \left\| \left(u(t_{i+1}, x)-u(t_i, x) \right) +  \sigma\left(u(t_i(\varepsilon), x)\right) \left(\widetilde{\mathcal D}_{\varepsilon}v\right)\left(t_i, x\right) \right\|_{L^2} \\
\leq \, & \, c_{3,8} N^{  H/2-1-\delta}   \sum_{i=0}^{N-1}   \left\| \left(u(t_{i+1}, x)-u(t_i, x) \right) +  \sigma\left(u(t_i(\varepsilon), x)\right) \left(\widetilde{\mathcal D}_{\varepsilon}v\right)\left(t_i, x\right)  \right\|_{L^2},
\end{split}
\end{equation}
where  $\delta$ is the constant in $\mathcal I$ given by \eqref{eq main 2}.

Since $\left(\widetilde{\mathcal D}_{\varepsilon}v\right)(t, x)$ is independent of the sigma-algebra generated by $\{u(s, x),\, 0 \le s \le t_i(\varepsilon)\}$,       by Minkowski's inequality,  Theorem \ref{them existence}, and \eqref{eq v  holder}, we have that  for  $ i = 0, 1, \cdots, N,$
\begin{equation}\label{eq B1N-2}
\begin{split}
 & \left\| \left(u(t_{i+1}, x)-u(t_i, x) \right) +  \sigma\left(u(t_i(\varepsilon), x)\right) \left(\widetilde{\mathcal D}_{\varepsilon}v\right)\left(t_i, x\right)  \right\|_{L^2}  \\
    \leq \, & \,  \left\| u(t_{i+1}, x)-u(t_i, x) \right\|_{L^2} +  \left\| \sigma \left(u(t_i(\varepsilon), x)\right)\right\|_{L^2}\cdot   \left\|  \left(\widetilde{\mathcal D}_{\varepsilon}v\right)\left(t_i, x\right)  \right\|_{L^2}  \\
   \le \, & \, c_{6,5} N^{- H/2},
\end{split}
\end{equation}
where $c_{6,5}>0$.

Putting \eqref{eq B1N-1} and \eqref{eq B1N-2} together, we obtain that 
\begin{equation}\label{eq bound B 1 N}
    \mathbb{E}\left[|L_{1, N}|\right] \leq c_{3,8}c_{6,5} N^{- \delta}.
\end{equation}
We estimate  $ L_{2,N}$ as follows 
	\begin{align*}
		\mathbb{E} \left[L_{2,N}^{2}\right]  =& \,	\mathbb{E} \Bigg[\sum_{i,j=0} ^{N-1} \sigma^{2}   (u(t_{i}(\varepsilon),x)\sigma  ^{2} \left(u(t_{j}(\varepsilon),x\right) \left( N ^{H-1}\left(\widetilde{\mathcal D}_{\varepsilon}v\right)^2\left(t_i, x\right)- \kappa^{2} N^{-1}\right) \\
		 & \ \  \ \cdot  \left( N ^{H-1}\left(\widetilde{\mathcal D}_{\varepsilon}v\right)^2\left(t_j, x\right)- \kappa^{2} N^{-1}\right)\Bigg]\\
		 =& \,\sum_{i=0} ^{N-1} 	\mathbb{E} \left[ \sigma^{4}   (u(t_{i}(\varepsilon),x)\right]\cdot	\mathbb{E}\left[\left( N ^{H-1}\left(\widetilde{\mathcal D}_{\varepsilon}v\right)^2\left(t_i, x\right)- \kappa^{2} N^{-1}\right)^{2}\right] \\
		 & \ \ \ + 2	\mathbb{E} \Bigg[\sum_{0\le j <i\le N-1}\sigma^{2}   (u(t_{i}(\varepsilon),x)\sigma  ^{2} (u(t_{j}(\varepsilon),x)\left( N ^{H-1}\left(\widetilde{\mathcal D}_{\varepsilon}v\right)^2\left(t_i, x\right)- \kappa^{2} N^{-1}\right)\\
		 & \ \ \ \ \ \cdot \left(N ^{H-1}\left(\widetilde{\mathcal D}_{\varepsilon}v\right)^2\left(t_j, x\right)- \kappa^{2} N^{-1}\right)\Bigg]\\
		  =:&\,  l_{2,N}^{(1)}+ l_{2,N}^{(2)},
	\end{align*}
	where we used the fact that $\left(\widetilde{\mathcal D}_{\varepsilon}v\right)\left(t_i, x\right)$ is independent of $ \sigma \{ u(s, x), s\leq t_{i}(\varepsilon), x\in \mathbb{R}\}$.  By the Lipschitz assumption on $\sigma$,  \eqref{eq bound}, and the bound (\ref{2f-2}) in Lemma \ref{ll2},
	\begin{align*}
		l_{2,N}^{(1) }
		\leq   \,  c_{6,6} N ^{-1},
	\end{align*}
	where $c_{6,6}>0$.
	
	 In order to deal with the summande
	 $ l_{2,N}^{(2)}$, we will consider the following situations:
	 	\begin{itemize}
		\item When $ i-j \leq   \lfloor N ^{-\theta+1} \rfloor+1$.  The corresponding sum will be denoted by $ l_{2, N} ^{(2,1)}$.
		 		\item  When  $ i-j \geq \lfloor N ^{-\theta+1} \rfloor+2$. In this case we have
		$$t_{j}+ \varepsilon \leq t_{i} (\varepsilon).$$
		This implies that  $\left(\widetilde{\mathcal D}_{\varepsilon}v\right)\left(t_i, x\right)$  and $\left(\widetilde{\mathcal D}_{\varepsilon}v\right)^2\left(t_j, x\right)$  are independent of  $\sigma^{2}   (u(t_{i}(\varepsilon),x)$ and $\sigma  ^{2} (u(t_{j}(\varepsilon),x)$. The corresponding sum will de denoted by $ l_{2,N} ^{(2,2)}$. 
	\end{itemize}
Here,  	 $\lfloor x \rfloor$  denotes  the biggest integer less or equal than $x$. 	

Let us now estimate the quantities $ l_{2,N} ^{(2,1)}$ and $ l_{2, N} ^{(2,2)}$. By the Cauchy-Schwarz inequality, the Lipschitz condition of $\sigma$,   \eqref{eq bound}, and the hypercontractivity property   \cite[(26)]{OT25}, we have 
	\begin{eqnarray*}
		l_{2, N} ^{(2,1)}&\leq & \sum_{i,j=0; 0<i-j\leq \lfloor N ^{-\theta+1} \rfloor+1}^{N-1}   \left\{ 	\mathbb{E}\left[  \sigma^{8}   (u(t_{i}(\varepsilon),x))\right]\right\} ^{\frac{1}{4}}  \left\{ 	\mathbb{E} \left[ \sigma^{8}  \left(u(t_{j}(\varepsilon),x)\right)\right]\right\} ^{\frac{1}{4}} \\&&
		\ \ \ \ \  \cdot \left\{ 	\mathbb{E} \left[ \left( N ^{H-1}\left(\widetilde{\mathcal D}_{\varepsilon}v\right)^2\left(t_i, x\right)- \kappa^{2} N^{-1}\right)^{4}\right] \right\}^{\frac{1}{4}} \\
		&& \ \ \ \ \ \cdot \left\{ 	\mathbb{E}\left[  \left( N ^{H-1}\left(\widetilde{\mathcal D}_{\varepsilon}v\right)^2\left(t_j, x\right)- \kappa^{2} N^{-1}\right)^{4} \right]\right\} ^{\frac{1}{4}}\\
		&\leq & c_{6,7} \sum_{i,j=0; 0< i-j\leq \lfloor N ^{-\theta+1} \rfloor+1}^{N-1} \left( 	\mathbb{E} \left( N ^{H-1}\left(\widetilde{\mathcal D}_{\varepsilon}v\right)^2\left(t_i, x\right)- \kappa^{2} N^{-1}\right)^{2} \right) ^{\frac{1}{2}} \\
		&& \ \ \ \ \  \ \ \ \  \  \ \ \ \ \cdot  \left( 	\mathbb{E} \left( N ^{H-1}\left(\widetilde{\mathcal D}_{\varepsilon}v\right)^2\left(t_j, x\right)- \kappa^{2} N^{-1}\right)^{2} \right) ^{\frac{1}{2}},
	\end{eqnarray*}
	where $c_{6,7}>0$. Consequently,   	by using the inequality (\ref{2f-2}) in Lemma \ref{ll2} and the fact that $(\widetilde{\mathcal D}_{\varepsilon} v)(t, x)$ equals in distribution to 
$({\mathcal D}_{\varepsilon} v)(t, x)$, we obtain
	 	\begin{equation}\label{eq L2N21}
		l_{2, N} ^{(2,1)}\leq c_{6,8}    N ^{-\theta}. 
			\end{equation}
		for some constant $c_{6,8}>0$.
		
	For the summand   $l_{2,N} ^{(2,2)}$, since  $\left(\widetilde{\mathcal D}_{\varepsilon}v\right)\left(t_i, x\right)$ and $\left(\widetilde{\mathcal D}_{\varepsilon}v\right)^2\left(t_j, x\right)$ are independent  of the other random variables in the sum, \eqref{eq bound} and the Lipschitz continuity of   $\sigma$ imply that
	 	 	\begin{align*}
		l_{2,N} ^{(2,2)} 
		 =&\,  2\sum_{i,j=0; i-j\geq \lfloor N ^{-\theta+1} \rfloor+2}^{N-1} 
		\mathbb{E} \left[ \sigma^{2}   (u(t_{i}(\varepsilon),x))\sigma  ^{2} (u(t_{j}(\varepsilon),x))\right]\\	
		 &\ \ \ \ \ \cdot \mathbb{E}\left[\left( N ^{H-1}\left(\widetilde{\mathcal D}_{\varepsilon}v\right)^2\left(t_i, x\right)- \kappa^{2} N^{-1}\right)\left( N ^{H-1}\left(\widetilde{\mathcal D}_{\varepsilon}v\right)^2\left(t_j, x\right)- \kappa^{2} N^{-1}\right)\right] \\
		 \leq & \, c_{6,9} \sum_{i,j=0; i-j\geq \lfloor N ^{-\theta+1} \rfloor+2}^{N-1}   \mathbb{E}\left[\left|\left( N ^{H-1}\left(\widetilde{\mathcal D}_{\varepsilon}v\right)^2\left(t_i, x\right)- \kappa^{2} N^{-1}\right)\right. \right.\\
		 &\ \ \ \ \ \ \ \ \ \ \ \ \ \ \cdot  \left. \left. \left( N ^{H-1}\left(\widetilde{\mathcal D}_{\varepsilon}v\right)^2\left(t_j, x\right)- \kappa^{2} N^{-1}\right)\right|\right],
	\end{align*}
    where $c_{6,9}\in (0,\infty)$.
    
	Now, we use the estimate (\ref{14s-1})  in Lemma \ref{ll2} to get that
	\begin{equation}\label{27i-1}	
	\begin{split}
		l_{2,N} ^{(2,2)} \leq &\, c_{6,2}c_{6,9} N^{-2}  \left( \sum_{i,j=1; i-j\geq \lfloor N ^{-\theta+1} \rfloor+2}^{N-1} \left( \rho ^{2}_{ H}(i-j)+ i ^{H/2-1} \right)+1\right) \\
		 \leq &\,c_{6,2}c_{6,9} N^{-2} \left( \sum_{i,j=1; i-j\geq1}^{N-1} \left( \rho ^{2}_{ H}(i-j)+ i ^{H/2-1} \right)+1\right) \\
		 \leq & \, c_{6,2}c_{6,9} \left( N ^{-2} \sum _{j=1} ^{N} \sum _{i=j+1} ^{N} \rho ^{2}_{ H}(i-j) + N^{-1} \sum_{i=1} ^{N-1} i ^{H/2-1} +N^{-2}\right)  \\
		 =&\, c_{6,2}c_{6,9} \left(N^{-2} \sum _{j=1} ^{N} \sum_{k=1} ^{N-1-j} \rho ^{2}_{ H}(k) +N^{-1} \sum_{i=1} ^{N-1} i ^{H/2-1} + N^{-2}\right) \\
		 \leq &\,  c_{6,10}  \left( N^{-1} \sum_{k=1}^{N-1}\rho^{2} _{H}(k) + N^{-1} \sum_{i=1} ^{N-1} i ^{H/2-1} + N^{-2}\right) \\
		 \leq & \,c_{6,11} \left( N ^{-1} + N ^{  \frac H2-1}+N^{-2}\right),
		 	\end{split}
	\end{equation}
	where   $c_{6,10}, c_{6,11}\in (0,\infty)$,  and  the fact that $\sum_{k=1}^{\infty}\rho^{2} _{H}(k)<\infty$ is used in the last step.  
	
	Putting \eqref{eq L2N21} and \eqref{27i-1} together, 	we have  that by Jensen's inequality, 
	\begin{equation}\label{3f-2}
		  \mathbb{E} \left[\vert L_{2,N}\vert \right]  \leq c_{6,12} N ^{-\frac12\min\{\theta, 1-\frac{H}{2}\}},
	 \end{equation}
where $c_{6,12}>0$.

	By \eqref{eq Holder} and \eqref{eq tiN}, we have  
	\begin{equation}\label{3f-12}
	\begin{split}
		\mathbb{E}\left[ \vert L_{3, N}  \vert \right] \leq & \, \kappa^{2} \sum_{i=0} ^{N-1} \int_{t_{i}}^{t_{i+1}}   	  \left\|   \sigma (u(t_{i}(\varepsilon), x))-\sigma (u(s,x)) \right\|_{L^2} ds \\
		 \leq &\,    \kappa^{2}   L_{\sigma}\sum_{i=0} ^{N-1} \int_{t_{i}}^{t_{i+1}}  	 \left\|u(t_{i}(\varepsilon), x)-u(s,x)  \right\|_{L^2}ds \\
		 \leq & \,c_{6,13}    \sum_{i=0} ^{N-1} \int_{t_{i}}^{t_{i+1}}\vert t_{i}(\varepsilon)-s\vert ^{\frac{H}{2 }}ds\\
		 \leq & \, c_{6,13}     N^{-\frac{\theta H}{2}},
		\end{split}
	\end{equation}
	 where $c_{6,13}>0$.
	 
	By combining the estimates \eqref{eq bound B 1 N},  \eqref{3f-2},  and  \eqref{3f-12},  we  obtain the inequality \eqref{eq quad}. The proof is complete. 
    \end{proof}

Using the same approach as in   \cite[Section 3.3]{OT25}, one can obtain the asymptotic behavior of 
 the higher order variations of the temporal variation for the solution $\{u(t, x)\}_{t \geq 0}$. Here, 
  we state a result of the asymptotic behavior of the weighted temporal variation,
 by using the same approach as in the proof of \cite[Corollary 1.6]{WX2024}.
\begin{proposition}\label{coro variations}
        Assume that Condition \ref{cond A} holds. If $\varphi:\mathbb R
        \rightarrow\mathbb R$
        is Lipschitz continuous,  then for all    $x\in \mathbb R$,
        \begin{equation*}
        \begin{split}
        &\lim_{n\rightarrow\infty} \sum_{0\le j \le   2^n}\varphi\left(u\left(\frac{j}{2^n}, x\right)\right)\left|  u\left(\frac{j+1}{2^n}, x\right)-
        u\left(\frac{j}{2^n}, x\right) \right|^{\frac{2}{H}} \\
        =& \,    c_{6,14}   \int_{0}^{1}\varphi(u(t, x)) |\sigma(u(t, x))|^{\frac{2}{H}}dt,
                \end{split}
        \end{equation*}
        almost surely and in $L^2(\Omega, \mathbb P)$, where $c_{6,14}= \kappa^{\frac{2}{H}}
        \mathbb E\big[|\mathcal N|^{\frac{2}{H}}\big]$ with   $\mathcal N$ being a standard Gaussian random variable.
        \end{proposition}

  Propositions \ref{prop variation converge} and \ref{coro variations} characterize the temporal quadratic and higher-order variations of solutions to the following parametrized stochastic heat equation:
 \begin{equation}\label{eq parametrized SFBE}
    \frac{\partial  u_\theta (t, x)}{\partial t} =  \theta \Delta u_\theta(t, x)  +\sigma(u_\theta (t, x)) \dot W(t, x), \ \ \ t>0, x\in \mathbb R,
    \end{equation}
    with vanishing initial value  $u_\theta(0, x)=0$ for every $x\in \mathbb R$.  
Consequently, one can then construct an estimator for the drift parameter $\theta$ using discrete-time observations of $u_\theta$ at an arbitrary spatial point. The procedure follows standard arguments; see \cite{GT2023, LSY25,  OT25, PT2007}  for details.
 
     \appendix
\section{Some lemmas}


      \begin{lemma}\label{lem int p}$($\cite[Lemma 3.1]{HHLNT2017}$)$ For any $t>0$ and $\beta\in (0,1)$, we have
    $$
    \int_{\mathbb R}  \int_{\mathbb R} |p_t(x+h)-p_t(x)|^2 |h|^{-1-2\beta}dh  dx   = c  t^{-\frac12-\beta},
    $$
where $c$ is a positive constant.
     \end{lemma}

    \begin{lemma}\label{lem green 2} For any $H\in \left(\frac14, \frac12\right)$ and   $\theta\in \left(0, \frac12- H\right)$, we have
    $$
    \int_{\mathbb R}   \int_{\mathbb R}    \Big|e^{-(x+h)^2}-e^{-x^2}\Big|^2  |h|^{2H-2} |x|^{2\theta} dhdx <\infty.
    $$
    \end{lemma}
         \begin{proof} By \cite[Lemma 2.10]{HW2022}, we know that for any $H\in \left(\frac14, \frac12\right)$, there exists a positive constant $c_H$ satisfying that
     $$
     \int_{\mathbb R}     \Big|e^{-(x+h)^2}-e^{-x^2}\Big|^2  |h|^{2H-2}  dh\le c_H \left(1\wedge  |x|^{2H-2}  \right).
    $$
          Consequently,   for any $\theta\in \left(0, \frac12-H\right)$,
    \begin{align*}
    &\int_{\mathbb R}dx   |x|^{2\theta}    \int_{\mathbb R} dh     \Big|e^{-(x+h)^2}-e^{-x^2}\Big|^2  |h|^{2H-2} \\
    \le &\,  c_H  \int_{\mathbb R} |x|^{2\theta}    \left(1\wedge  |x|^{2H-2}  \right)  dx<\infty.
     \end{align*}
      The proof is complete.
    \end{proof}

The next lemma is an isoperimetric inequality for general Gaussian processes.

\begin{lemma}\label{Lem:LT}$($\cite[p.302]{LT}$)$
There is a universal constant $K_0$ such that the following statement holds.
Let $S$ be a bounded set and $\{X(s), s \in S\}$ be a separable Gaussian process.
Let $D = \sup\{ d_X(s, t) : s, t \in S \}$ be the diameter of $S$ in metric $d_X$.
Then for any $u > 0$,
\begin{equation*}
\mathbb P\bigg\{ \sup_{s, t \in S} |X(s) - X(t)| \ge 
K_0\Big( u + \int_0^D \sqrt{\log N(S, d_X, \eps)} d\eps \Big) \bigg\} 
\le \exp\left(-\frac{u^2}{D^2}\right).
\end{equation*}
\end{lemma}

\begin{lemma}$($\cite[Theorem 2.4]{Tal94}\label{Tal94}$)$
Let $\{ X(t), t \in S \}$ be a mean-zero continuous Gaussian process.
Denote
$\sigma^2 := \sup_{t \in S}\|X(t)\|_{L^2}^2$  and  
$d_X(s, t) := \|X(s) - X(t)\|_{L^2}$.
Assume that  there exist   some constants $M > \sigma$,   $p > 0$, and   $0 < \eps_0 \le \sigma$  such  that 
\begin{equation*}
N(S, d_X, \eps) \le (M/\eps)^p \quad \text{for all } \eps < \eps_0.
\end{equation*}
Then, it holds that for any  $u > \sigma^2[(1+\sqrt{p})/\eps_0]$,  
\begin{equation*}
\mathbb P\bigg\{ \sup_{t \in S} X(t) \ge u \bigg\} 
\le \left( \frac{KMu}{\sqrt{p}\,\sigma^2} \right)^p \Phi\left( \frac{u}{\sigma} \right),
\end{equation*}
where $\Phi(x) = (2\pi)^{-1/2} \int_x^\infty e^{-y^2/2} \, dy$
and $K$ is a universal constant.
\end{lemma}

The following Gaussian estimate is standard:
\begin{equation}\label{stdGaussian}
\frac{1}{2\sqrt{2\pi}x} e^{-x^2/2} \le \Phi(x) \le 
\frac{1}{\sqrt{2\pi}} e^{-x^2/2} \quad \text{for all } x \ge 1.
\end{equation}

     \vskip0.3cm


 \vskip0.3cm
 
\noindent{\bf Acknowledgments}:  We thank Prof. Fuqing Gao for many fruitful discussions. The research of R. Wang is partially supported by the NSF of Hubei Province (Grant No. 2024AFB683). The research of Y. Xiao is partially supported by the NSF grant DMS-2153846.

    \vskip0.5cm


\vskip0.8cm


\begin{thebibliography}{abc}
     \bibitem{BJQ2015}
    Balan, R.,  Jolis,  M., Quer-Sardanyons, L.: SPDEs with affine multiplicative fractional noise in space with index $\frac {1}{4}< H<\frac {1}{2}$. 
     {\it Electron. J. Probab.} \textbf{20}(54), 1--36 (2015)

    \bibitem{BJQ2016}
    Balan, R., Jolis, M., Quer-Sardanyons, L.: SPDEs with rough noise in space: H\"older continuity of the  solution.  {\it Statist. Probab. Lett.} 
    \textbf{119}, 310--316 (2016)

    \bibitem{CD15}
  Chen, L.,   Dalang,  R. C.:  Moments and growth indices for the nonlinear stochastic heat equation with rough initial
conditions. {\it Ann. Probab.} {\bf 43}(6), 3006--3051  (2015)
 
\bibitem{CHKK19}
 Chen, L.,  Huang, J.,   Khoshnevisan, D.,  Kim, K.: Dense blowup for parabolic SPDEs. {\it Electron. J. Probab.} {\bf 24},  1--33 (2019)

    \bibitem{CK2019}
Chen, L.,     Kim, K.:  Nonlinear stochastic heat equation driven by spatially colored noise: moments and intermittency.
 {\it Acta Math. Sci. Ser. B.} {\bf39}(3), 645--668  (2019)




    \bibitem{DaPrato2014} Da Prato, G., Zabczyk, J.: Stochastic equations in infinite dimensions.  Second edition. Cambridge University Press, Cambridge (2014)


    \bibitem{Dalang1999}
    Dalang, R. C.: Extending the martingale measure stochastic integral with applications to spatially homogeneous s.p.d.e.'s.
    {\it Electron. J. Probab.}  \textbf{4}(6), 1--29 (1999)
    
\bibitem{DMX17}
  Dalang,  R. C.,  Mueller,  C.,  Xiao, X.: Polarity of points for Gaussian random fields. {\it Ann. Probab.} 
{\bf 45}(6), 4700--4751  (2017)

\bibitem{DNP2025}
Dalang, R. C., Nualart, D.,  Pu, F.: Sharp upper bounds on hitting probabilities for the solution to the stochastic heat equation on the line. Preprint available at    \href{https://doi.org/10.48550/arXiv.2508.11859}{arXiv:2508.11859} (2025)


    \bibitem{DQ} Dalang, R. C.,  Quer-Sardanyons, L.: Stochastic integrals for spde's: A comparison. {\it Expo. Math.}  {\bf 29}(1), 67--109 (2011)

    \bibitem{Das2022}  Das,   S.: Temporal increments of the KPZ equation with general initial data.   {\it Electron. J. Probab.}  {\bf 29},     1--28 (2024)

    \bibitem{FKM2015} Foondun,  M.,    Khoshnevisan, D.,   Mahboubi, P.:  Analysis of the gradient of the solution to a stochastic heat equation via fractional Brownian motion. {\it Stoch PDE: Anal. Comp.}    {\bf 3}, 133--158  (2015)



     \bibitem{GT2023}
      Gamain, J., Tudor, C. A.:  Exact variation and drift parameter estimation for the nonlinear fractional stochastic heat equation. {\it Jpn. J. Stat. Data Sci.}  {\bf6}, 381--406 (2023)
      
\bibitem{HP15}
    Hairer, M.,   Pardoux, \'E.: A Wong–Zakai theorem for stochastic PDEs. {\it J. Math.
Soc. Japan} {\bf 67}(4), 1551--1604 (2015)

    \bibitem{HRZ2021} Hao, Z.,  R\"ockner, M.,  Zhang, X.-C.: Euler scheme for density dependent stochastic differential equations.   {\it J. Differential Equations}  {\bf 274}, 996--1014 (2021)




    \bibitem{HSWX20}
      Herrell, R.,   Song,  R.,  Wu, D.,    Xiao, Y.: Sharp space-time regularity
    of the solution to a stochastic heat equation driven by a fractional-colored noise.
    {\it Stoch. Anal. Appl.} {\bf 38}, 747--768 (2020)

\bibitem{HL25}
  Hu, J.,  Lee C. Y.:
On the spatio-temporal increments of nonlinear parabolic SPDEs and the open KPZ equation.   Preprint available at  
  \href{https://arxiv.org/pdf/2508.05032}{arXiv:2508.05032} (2025) 

    \bibitem{hu19}
     Hu, Y.: Some recent progress on stochastic heat equations.
    {\it Acta Math. Sci. Ser. B}   \textbf{39}(3),  874--914 (2019)


    \bibitem{HHLNT2017}
    Hu, Y., Huang, J.,  L\^e, K.,  Nualart, D., Tindel, S.: Stochastic heat equation with rough dependence in space.  {\it Ann. Probab.}    \textbf{45}(6),   4561--616  (2017)

    \bibitem{HHLNT2018}
     Hu, Y.,  Huang, J.,  L\^e, K.,  Nualart, D., Tindel,  S.: Parabolic Anderson model with rough dependence in space. Computation and combinatorics in dynamics, stochastics and control, 477--498. Abel Symp., 13, Springer, Cham (2018)



    \bibitem{HW2022}
    Hu, Y.,  Wang, X.: Stochastic heat equation with general rough noise. {\it Ann. Inst. Henri Poincar\'e Probab. Stat.}  \textbf{58}(1),  379--423 (2022)

    \bibitem{HK2017}
 Huang,  J.,   Khoshnevisan,  D.:  On the multifractal local behavior of parabolic stochastic PDEs. {\it Electron.
Commun. Probab.} {\bf 22},  1--11 (2017)

    \bibitem{KT2019}
   Khalil,    Z. M.,    Tudor,   C. A.:  On the distribution and $q$-variation of the solution to the heat equation with fractional
    Laplacian.  {\it Probab.  Math. Statist.} {\bf 39}(2), 315--335 (2019)


 \bibitem{K2014}
   Khoshnevisan, D.: {\it Analysis of stochastic partial differential equations}.  American Mathematical Soc.  (2014)

\bibitem{KKM23}
Khoshnevisan, D.,  Kim, K.,   Mueller, C.: Small-ball constants, and exceptional flat points of SPDEs. {\it Electron. J. Probab.}  {\bf 29}, 1--31 (2024) 

     \bibitem{KS2023}
      Khoshnevisan,    D.,  Sanz-Sol\'e, M.:  Optimal regularity of SPDEs with additive noise.  {\it Electron. J. Probab.}  {\bf 28}, 1--31  (2023)


    \bibitem{KSXZ2013}
       Khoshnevisan, D.,   Swanson,  J.,    Xiao,  Y.,     Zhang, L.: Weak existence of a solution to
    a differential equation driven by a very rough fBm.  Preprint available at
     \href{https://arxiv.org/pdf/1309.3613.pdf}{arXiv:1309.3613v2}  (2014) 
     
     
     \bibitem{LT}
 Ledoux, M.,   Talagrand M.: {\it Probability in Banach spaces}. Springer (1991)

\bibitem{LX2023}
Lee, C. Y.,   Xiao, Y.: Chung-type law of the iterated logarithm and exact moduli of continuity for a class of anisotropic Gaussian random fields. {\it Bernoulli} {\bf 29}(1), 523--550 (2023)



    \bibitem{LS01}
      Li,  W. V.,    Shao, Q.-M.:  Gaussian processes: inequalities, small ball probabilities and applications. In {\it Stochastic Processes:
    Theory and Methods. Handbook of Statistics}, 19, (C.R. Rao and D. Shanbhag, editors),   533--597, North-Holland (2001)
    
\bibitem{LSY25}
Li, Y., Shu, H.,  Yan, L.:    Exact temporal variation for fractional stochastic heat equation driven by space-time white noise.  Preprint available at  
  \href{https://doi.org/10.48550/arXiv.2410.20426}{arXiv:2410.20426} (2024)
  
\bibitem{LQW25}Liu, C., Qian, B.,  Wang, R.:
Growth rates for the  H\"older coefficients of   the linear stochastic fractional  heat equation  with  rough dependence in space.   Preprint available at  
  \href{https://arxiv.org/abs/2507.22379v1}{arXiv:2507.22379} (2025)  

\bibitem{LW25} Liu, C., Wang, R.:
Small-ball probabilities for the stochastic fractional  heat equation  with  rough dependence in space.  Preprint.	 



    \bibitem{LM2022}
      Liu, J.,   Mao, L.:  Nonlinear fractional stochastic heat equation driven by Gaussian noise rough in space. {\it Bull. Sci. Math.}   {\bf181},    103207 (2022)

 
    \bibitem{LHW2022}
    Liu, S., Hu, Y.,  Wang, X.:   Nonlinear stochastic wave equation driven by rough noise. {\it J. Differential Equations} {\bf 331}, 99--161  (2022)


 
   

 \bibitem{MR1995}
   Monrad,  D., Rootz\'en, H.:  Small values of Gaussian processes and functional laws of the iterated logarithm.
 {\it Probab. Theory Relat. Fields}   {\bf101}, 173--192  (1995)

    \bibitem{MT02}
     Mueller,  C., Tribe, R.:  Hitting probabilities of a random string. {\it Electronic J. Probab.} {\bf 7},   1--29 (2002)


		\bibitem{NP-book}
		 Nourdin, I.,    Peccati, G.: {\it  Normal Approximations with
			Malliavin Calculus From Stein's Method to Universality}. Cambridge
		University Press (2012)
		

\bibitem{OT25}
Olivera, C., Tudor, C.A.: Temporal quadratic and higher order variation for the nonlinear  stochastic heat equation  and applications to parameter estimation.  To appear in {\it Ann. Mat. Pura Appl.}  (2025)






    \bibitem{PT2007}  Posp\'i\v{s}il, J.,     Tribe, R.:  Parameter estimates and exact variations for stochastic heat equations
      driven by space-time white noise. {\it Stoch. Anal. Appl.} {\bf 25}(3), 593--611 (2007)

    \bibitem{Song2018}
    Song, J.: SPDEs with colored Gaussian noise: a survey. {\it  Commun. Math. Stat.}  \textbf{6}(4), 481--492 (2018)

   \bibitem{SongSX2020} Song, J., Song X., Xu, F.: Fractional stochastic wave equation driven by a Gaussian noise rough in space. {\it Bernoulli}  \textbf{26}(4), 2699-2726 (2020)



\bibitem{Tal94}
 Talagrand, M.: Sharper bounds for Gaussian and empirical processes. 
{\it Ann. Probab.} {\bf 22}(1), 28--76 (1994)
 
 \bibitem{Tal95}
   Talagrand, M.:  Hausdorff measure of trajectories of multiparameter fractional Brownian motion.  {\it Ann. Probab.} {\bf 23}, 767--775 (1995)



\bibitem{Tal96}
 Talagrand, M.:  Lower classes of fractional Brownian motion. {\it J. Theor. Probab.} {\bf 9}, 191--213 (1996)




 \bibitem{TX17}
 Tudor,  C. A.,  Xiao, Y.:  Sample path properties of the solution to the fractional-colored stochastic heat equation. {\it Stoch. Dyn.}
  {\bf 17}(1), 1750004 (2017)


      \bibitem{Wang2024}
       Wang,  R.:  Analysis of the gradient for the stochastic fractional heat equation with spatially-colored noise in  $\mathbb R$.
      {\it Discrete Contin. Dyn. Syst. (B)}   {\bf 29}(6),   2769--2785 (2024)

      \bibitem{WX2024}
      Wang, R.,    Xiao, Y.:
      Temporal properties of the stochastic  fractional heat equation with spatially-colored noise.  {\it Theory Probab. Math. Statist.}      {\bf 110}, 121--142 (2024)


     \bibitem{WZ2021}
     Wang,   R.,   Zhang, S.:  Decompositions of stochastic convolution driven by a white-fractional Gaussian noise. {\it Front.     Math. China} {\bf16}(4), 1063--1073  (2021)
     \end{thebibliography}
    \end{document}